\documentclass{article}
\usepackage[margin=1truein]{geometry}
\usepackage[hyphens]{url}
\usepackage{graphicx}
\usepackage{orcidlink}
\hypersetup{
    colorlinks=true,
    linkcolor=blue,
    citecolor=blue,
    urlcolor=blue
}
\usepackage{natbib}
\usepackage{caption}
\usepackage{algorithm}
\usepackage{algorithmic}

\usepackage{newfloat}
\usepackage{listings}
\DeclareCaptionStyle{ruled}{labelfont=normalfont,labelsep=colon,strut=off}
\floatstyle{ruled}
\newfloat{listing}{tb}{lst}{}
\floatname{listing}{Listing}

\usepackage{booktabs}

\usepackage{amsmath}
\usepackage{amssymb}
\usepackage{bm}
\usepackage{mathrsfs}
\usepackage{mathtools}
\usepackage{optidef}
\usepackage{physics}
\allowdisplaybreaks[4]

\newcommand{\relmiddle}[1]{\mathrel{}\middle#1\mathrel{}}

\usepackage{amsthm}
\usepackage{thm-restate}
\theoremstyle{plain}
\newtheorem{theorem}{Theorem}
\newtheorem{lemma}{Lemma}
\newtheorem{proposition}{Proposition}

\theoremstyle{definition}
\newtheorem{condition}{Condition}

\newtheorem{assumption}{Assumption}

\usepackage{afterpage}
\usepackage{multirow}
\usepackage{makecell}
\usepackage{subcaption}
\usepackage{tikz}
\usepackage{array}
\usepackage{tabularx}
\newcolumntype{C}[1]{>{\centering\arraybackslash}m{#1}}
\definecolor{cA}{HTML}{0072BD}
\definecolor{cB}{HTML}{EDB120}
\definecolor{cC}{HTML}{77AC30}
\definecolor{cD}{HTML}{D95319}

\usepackage{appendix}
\usepackage{enumitem}
\newlist{assumptionitems}{enumerate}{1}
\setlist[assumptionitems]{
    label=\arabic*,
    ref=\theassumption.\arabic*
}
\usepackage{pifont}
\usepackage{xparse}
\usepackage{siunitx}

\usepackage{cleveref}
\crefname{section}{Section}{Sections}
\crefname{theorem}{Theorem}{Theorems}
\crefname{proposition}{Proposition}{Propositions}
\crefname{lemma}{Lemma}{Lemmas}
\crefname{condition}{Condition}{Conditions}
\crefname{assumption}{Assumption}{Assumptions}
\crefname{assumptionitemsi}{Assumption}{Assumptions}
\Crefname{assumptionitemsi}{Assumption}{Assumptions}
\crefname{algorithm}{Algorithm}{Algorithms}
\crefname{table}{Table}{Tables}
\crefname{figure}{Figure}{Figures}
\newcommand{\refLine}[1]{Line~\ref{#1}}

\title{
    Adaptive Gradient-Based Methods for a Broader Class of\\
    Optimization Problems under Performative Prediction
}
\author{
    Hiroki Hamaguchi\,\orcidlink{0009-0005-7348-1356}\textsuperscript{\rm 1},
    Yuya Hikima\textsuperscript{\rm 2},
    Hiroshi Sawada\textsuperscript{\rm 3},
    Akiko Takeda\textsuperscript{\rm 1, \rm 4}
    \\[0.5em]
    \textsuperscript{\rm 1}Graduate School of Information Science and Technology, The University of Tokyo\\
    \textsuperscript{\rm 2}Department of Industrial Engineering and Economics, Institute of Science Tokyo\\
    \textsuperscript{\rm 3}Communication Science Laboratories, NTT, Inc.\\
    \textsuperscript{\rm 4}Center for Advanced Intelligence Project, RIKEN\\
}

\DeclareMathOperator{\Pred}{Pred}
\DeclareMathOperator{\Proj}{Proj}
\DeclareMathOperator*{\argmin}{arg\,min}
\newcommand{\Lsm}{L_{\mathcal{L}}^{\mathrm{sm}}}

\begin{document}

\maketitle

\begin{abstract}
    We study optimization under performative prediction, where deploying a model affects the future data distribution.
    For this setting, several gradient-based approaches have been proposed.
    However, they typically assume specific data distributions or loss functions,  which limit their practical applicability.
    To overcome these limitations, we propose a gradient-based optimization method with convergence guarantees under substantially weaker assumptions.
    Our method explicitly estimates the induced distribution shift through finite differences.
    It enables higher-dimensional optimization across broader classes of loss functions and data distributions.
    We also propose a practical variant that reduces the number of samples required.
    Numerical experiments demonstrate that our proposed algorithms converge faster and more consistently than existing ones.
\end{abstract}

\begin{center}
    \noindent\textbf{Code repository}: \url{https://github.com/HirokiHamaguchi/adaptive-gradient-methods-under-performative-prediction}
\end{center}

\section{Introduction}
\label{sec:intro}

In standard reinforcement learning and machine learning, the objective is often to minimize the expectation of a given loss function $\ell(z; \theta)$, where $\theta \in \Theta$ is the model parameter and the data point $z$ is drawn from an unknown data distribution $\mathcal{D}$.
Empirical risk minimization draws samples from $\mathcal{D}$ at training time, minimizes the empirical loss, and then deploys the resulting model for the test phase \citep{vapnikPrinciplesRiskMinimization1991}.
This approach assumes that both training and test data are drawn from the same fixed distribution independent of $\theta$.

More practical settings may involve \emph{distribution shifts} \citep{quinonero2008dataset} between training and test.
The data distribution observed during training, $\mathcal{D}_{\mathrm{train}}$, may differ from the one observed when the model is deployed, $\mathcal{D}_{\mathrm{test}}$.
Such discrepancies have been attributed to exogenous factors such as temporal evolution or environmental changes.

Recently, the framework of \emph{performative prediction} \citep{perdomoPerformativePrediction2020} has attracted significant attention as a distinct form of distribution shift, namely an endogenous shift induced by model deployment.
This framework models the feedback loop through which the deployed model parameter $\theta$ changes the data distribution $\mathcal{D}_{\beta(\theta)}$:
\begin{equation}\label{prob:performative_optimization}
    \underset{\theta \in \Theta}{\mathrm{minimize}} \quad
    \mathcal{L}(\theta) \coloneqq \mathbb{E}_{z \sim \mathcal{D}_{\beta(\theta)}}[\ell(z;\theta)],
\end{equation}
where $\beta(\theta)$ is the unknown distribution parameter induced by $\theta$, and $\mathcal{L}$ is the performative loss (risk).

Such endogenous distribution shifts appear in many applications.
Examples include pricing decisions that affect the distribution of market demand, spam-filter parameters that alter the distribution of attackers' behavior, and recommender-system parameters that shape the distribution of user preferences \citep{perdomoPerformativePrediction2020,izzoHowLearnWhen2021d,hardtPerformativePredictionFuture2025b}.
To understand the setting, consider a pricing problem.
Let $\theta_{\mathrm{profit}}$ denote the vector of profit margins, and let $z_{\mathrm{sales}}$ denote the vector of daily sales counts.
A seller chooses the prices $\theta_{\mathrm{profit}}$ to maximize expected daily profit, or equivalently, minimize the loss $\ell(z_{\mathrm{sales}};\theta_{\mathrm{profit}}) = -z_{\mathrm{sales}}^\top \theta_{\mathrm{profit}}$.
This results in
\begin{equation*}
    \underset{\theta_{\mathrm{profit}} \in \Theta}{\mathrm{minimize}} \quad
    \mathbb{E}_{z_{\mathrm{sales}} \sim \mathcal{D}_{\beta(\theta_{\mathrm{profit}})}}
    \left[
        - z_{\mathrm{sales}}^\top \theta_{\mathrm{profit}}
        \right],
\end{equation*}
which is an instance of Problem~\eqref{prob:performative_optimization}.

To solve Problem~\eqref{prob:performative_optimization}, it is important to accurately estimate the induced distribution shift, i.e., the change in $\mathcal{D}_{\beta(\theta)}$ as a function of $\theta$.
If this shift is ignored, optimization can update the decision variable in a direction that is desirable under the current distribution but harmful after deployment.
In the pricing example, this corresponds to raising prices while ignoring the induced decrease in sales volume.

\paragraph{Research Question}

\begin{table*}[t]
    \centering
    \caption{
        Comparison of methods.
        In the Convergence Guarantee column, \ding{51} indicates convergence under restrictive assumptions such as Gaussian data distributions and non-degeneracy of successive parameter changes, and \ding{51}\ding{51} indicates convergence under weaker assumptions.
        The Sample Complexity column reports sample complexities for finding an $\epsilon$-stationary point under an $n$-dimensional Gaussian distribution with fixed variance and a polynomial loss.
        PerfGD has better sample complexity than the other methods, but it requires stronger assumptions and may fail to converge under weaker assumptions.
        Our main complexity result is based on coordinate-wise finite differences (\cref{thm:main}).
        Complexities marked with $\dagger$ are based on sphere smoothing (\citep[Theorem~10]{hikimaZerothorderGradientEstimators2025} and \cref{prop:sphere_smoothing}).
    }
    \label{tab:method_comparison}
    {\scriptsize
        \setlength{\tabcolsep}{0pt} 
        \renewcommand{\arraystretch}{0.95} 
        \begin{tabular}{
                C{0.17\linewidth}
                C{0.15\linewidth}
                C{0.12\linewidth}
                C{0.12\linewidth}
                C{0.12\linewidth}
                C{0.12\linewidth}
                C{0.20\linewidth}
            }
            \toprule
            Algorithm
             & {\centering Data Distribution                                                \\ Estimation}
             & {\centering Adaptive                                                         \\ Data Sampling}
             & {\centering Exploits                                                         \\ Loss Function}
             & {\centering Non-Gaussian                                                     \\ Distributions}
             & {\centering Convergence                                                      \\ Guarantee}
             & Sample Complexity                                                            \\
            \midrule
            {\centering \textbf{RRM} / \textbf{RGD}                                         \\[-0.4em] \scalebox{0.8}{\citep{perdomoPerformativePrediction2020}}}
             & ignore
             & \ding{51}
             & \ding{51}
             & \ding{51}
             & --
             & converge to a stable point                                                   \\
            {\centering \textbf{Plug-in}                                                    \\[-0.4em] \scalebox{0.8}{\citep{linPluginPerformativeOptimization2024b}}}
             & offline
             & --
             & \ding{51}
             & \ding{51}
             & --
             & implementation dependent                                                     \\
            {\centering \textbf{DFO}                                                        \\[-0.4em] \scalebox{0.8}{\citep{hikimaZerothorderGradientEstimators2025}}}
             & --
             & \ding{51}
             & --
             & \ding{51}
             & \ding{51}\ding{51}
             & $\order{n^{2.5}\epsilon^{-5}}$ $\big(\order{n^2 \epsilon^{-5}}^\dagger\big)$ \\
            {\centering \textbf{PerfGD}                                                     \\[-0.4em] \scalebox{0.8}{\citep{izzoHowLearnWhen2021d}}}
             & history-based
             & \ding{51}
             & \ding{51}
             & --
             & \ding{51}
             & $\order{n^{1.5} \epsilon^{-4}}$ or fail                                      \\
            \textbf{Proposed}
             & finite difference
             & \ding{51}
             & \ding{51}
             & \ding{51}
             & \ding{51}\ding{51}
             & $\order{n^{2.5}\epsilon^{-5}}$ $\big(\order{n^2 \epsilon^{-5}}^\dagger\big)$ \\
            \bottomrule
        \end{tabular}
    }
\end{table*}

Among several existing methods, the gradient-based algorithm PerfGD \citep{izzoHowLearnWhen2021d} is particularly promising for performative optimization.
PerfGD uses historical trajectories to approximate how the data distribution changes after model updates.
However, the theoretical guarantees for PerfGD mainly cover a one-dimensional Gaussian location family with fixed variance, require lower bounds on the gradient norm, and assume non-degeneracy of successive parameter changes.
Guarantees under weaker assumptions remain open \citep[Section 4.1]{izzoHowLearnWhen2021d}.
Thus, the following research question naturally arises:

\vskip\baselineskip
\emph{Can we develop a gradient-based method for performative optimization with convergence guarantees beyond restrictive assumptions on data distributions and losses?}
\vskip\baselineskip

\paragraph{Our Contributions}

In this paper, we develop a gradient-based algorithm that adaptively estimates the distribution parameter shifts induced by deployed models and incorporates these estimates into gradient-based updates.
Our contributions are as follows.
\begin{enumerate}[label=(\arabic*)]
    \item We extend gradient-based performative optimization to broader classes of data distributions and loss functions.
          The scope includes smooth losses with at most polynomial growth, such as polynomial, logistic, and softmax cross-entropy losses, including degenerate losses.
          It also includes high-dimensional Gaussian, Bernoulli, Poisson, log-normal, and regular exponential-family distributions under the conditions stated in \cref{asm:main}.
    \item We establish theoretical guarantees under these general conditions.
          In particular, we prove the $\Lsm$-smoothness of the performative loss $\mathcal{L}$, derive convergence guarantees, and obtain a sample complexity that matches the best existing result under these weak assumptions.
    \item We demonstrate stable empirical performance across a range of performative scenarios through experiments.
          Our method and its practical heuristic variant outperform existing approaches, providing empirical evidence for the benefits of adaptive gradient-based estimation.
\end{enumerate}
These contributions extend the applicability of gradient-based methods beyond previous work \citep{izzoHowLearnWhen2021d}, answering the research question posed above.

\paragraph{Related Work}

Let us briefly review the existing approaches to performative optimization with \cref{tab:method_comparison}.
Repeated Risk Minimization (RRM) and Repeated Gradient Descent (RGD) are simple retraining-based methods that ignore the distribution shift.
Under suitable conditions, they converge to performatively stable points ($\theta_{\mathrm{PS}} \in \argmin_{\theta \in \Theta} \mathbb{E}_{z \sim \mathcal{D}_{\beta(\theta_{\mathrm{PS}})}}[\ell(z;\theta)]$), which are invariant under retraining \citep{perdomoPerformativePrediction2020}.
Plug-in methods first estimate a parametric model of the dependence $\theta \mapsto \mathcal{D}_{\beta(\theta)}$ and then optimize the performative risk induced by the estimated model \citep{linPluginPerformativeOptimization2024b,millerOutsideEchoChamber2021}.
DFO methods, also known as zeroth-order methods, treat the performative risk as a black-box objective, requiring less modeling structure \citep{flaxmanOnlineConvexOptimization2004,hikimaGuidedZerothOrderMethods2025,hikimaZerothorderGradientEstimators2025,hikimaZerothOrderMethodsNonconvex2025}.

Still, RRM/RGD may not yield the optimal solution, plug-in methods are vulnerable to misspecification of the data distribution, and DFO methods do not exploit the known loss function and distribution-family structure in Problem~\eqref{prob:performative_optimization}.

We note that gradient-based methods, including our proposed method, can overcome these limitations.
As summarized in \cref{tab:method_comparison}, gradient-based methods can naturally exploit the known structure of the loss function and distribution family while adaptively estimating the unknown distribution parameter mapping.
As in \cref{tab:method_comparison}, we focus only on sample complexity because the cost of environment interaction typically dominates computational cost in practice.
Taken together, these observations highlight the advantages of our approach over existing methods.

\paragraph{Notation}

In this paper, $\mathbb{N}$ denotes the positive integers, $\mathbb{R}$ denotes the real numbers, $\norm{\cdot}$ denotes the Euclidean norm for vectors and induced norm for matrices, and $C^{k}$ denotes the set of $k$-times continuously differentiable functions.

\section{Preliminaries}

We introduce the setting.
Let $\Theta \subseteq \mathbb{R}^n$ be a nonempty closed convex feasible set for the model parameter, and $\theta \in \Theta$ be an $n$-dimensional model parameter.
Let $\beta\colon \mathbb{R}^n \to \mathbb{R}^m$ be a differentiable distribution parameter mapping from a model parameter $\theta$ to an $m$-dimensional distribution parameter $\beta(\theta)$.
In performative prediction, the $d$-dimensional data distribution $\mathcal{D}_{\beta(\theta)}$ depends on $\theta$ through $\beta$.
Our goal is to optimize the differentiable loss function $\ell \colon \mathbb{R}^d \times \mathbb{R}^n \to \mathbb{R}$ under the distribution induced by the deployed model.

Next, we introduce additional notation.
Let $B \subseteq \mathbb{R}^m$ denote the distribution parameter space for $\beta(\theta)$.
For $\beta_0 \in B$ and a measure $\mu$ on the data space,%
\footnote{%
    It is the Lebesgue measure $\dd{\mu(z)}=\dd{z}$ for continuous distributions, and the counting measure for discrete distributions.
}
let $p(z;\beta_0)$ denote the probability density function or probability mass function of $z \sim \mathcal{D}_{\beta_0}$, which satisfies $\int p(z;\beta_0) \dd{\mu(z)} = 1$.
See also Appendix~\ref{app:exact_formulation_of_assumptions}.
As we will see later in \cref{item:regularity}, we also assume that we can exchange integration and differentiation.

\paragraph{Unknown Distribution Parameter Mapping}

A key aspect of our setting is that the optimizer knows the loss function and the distribution family, but does not know the distribution parameter mapping $\beta(\theta)$.
In the pricing example, the optimizer knows the loss $\ell$ and models the sales counts as Poisson random variables, so $\mathcal{D}_{\beta(\theta_{\mathrm{profit}})}$ is the corresponding multivariate Poisson distribution.
Still, since $\beta(\theta_{\mathrm{profit}})$ is the vector of expected sales rates induced by the chosen profit margins $\theta_{\mathrm{profit}}$, the optimizer does not know the exact mapping form.
This information asymmetry is typical in performative prediction, motivating us to leverage the known structure of the loss and distribution while adaptively estimating the unknown distribution parameter mapping.

As an estimator of the distribution parameter $\beta(\theta)$, we assume the existence of $\Pred(\{z_j\}_{j=1}^{b})$ that takes a dataset of i.i.d.\ samples drawn from $\mathcal{D}_{\beta(\theta)}$ and outputs $\hat\beta \in B$ approximating $\beta(\theta)$.
In practice, $\beta(\theta)$ consists of quantities such as expectations, variances, or other moments of the distribution.
We can use the corresponding sample statistics as $\Pred$.

\paragraph{Structure of Gradient}

To optimize the performative loss $\mathcal{L}(\theta)$, we need to compute its gradient.
We view the performative loss $\mathcal{L}(\theta)$ as a special case of a more general two-argument function $\tilde{\mathcal{L}}(\theta_1,\theta_2)$, defined as follows:
\begin{equation*}
    \tilde{\mathcal{L}}(\theta_1,\theta_2) \coloneqq \mathbb{E}_{z \sim \mathcal{D}_{\beta(\theta_2)}}[\ell(z;\theta_1)]
                                           = \int \ell(z; \theta_1) p(z;\beta(\theta_2)) \dd{\mu(z)}.
\end{equation*}
Since $\mathcal{L}(\theta)=\tilde{\mathcal{L}}(\theta,\theta)$, we can compute the gradient of the performative loss \citep{izzoHowLearnWhen2021d} as
\begin{equation}
    \nabla \mathcal{L}(\theta) = \nabla_1\mathcal{L}(\theta) + \nabla_2\mathcal{L}(\theta) \label{eq:nabla_L}
\end{equation}
where
\begin{align}
    \!\nabla_1\mathcal{L}(\theta) & \coloneqq \nabla_{\theta_1} \tilde{\mathcal{L}}(\theta_1,\theta_2) \eval_{\theta_1=\theta_2=\theta}
                                  &&= \int \nabla_\theta \ell(z; \theta) p(z;\beta(\theta)) \dd{\mu(z)}
                                  = \mathbb{E}_{z \sim \mathcal{D}_{\beta(\theta)}} \qty[\nabla_\theta \ell(z; \theta)],      \label{eq:nabla_L1} \\
    \!\nabla_2\mathcal{L}(\theta) & \coloneqq \nabla_{\theta_2} \tilde{\mathcal{L}}(\theta_1,\theta_2) \eval_{\theta_1=\theta_2=\theta}
                                  &&= \int \ell(z; \theta) \nabla_\theta p(z;\beta(\theta)) \dd{\mu(z)}
                                  = {\pdv{\beta}{\theta}}(\theta)^\top \int \ell(z; \theta) \nabla_\beta p(z;\beta(\theta)) \dd{\mu(z)} \notag \\
                                  &&& = {\pdv{\beta}{\theta}}(\theta)^\top \mathbb{E}_{z \sim \mathcal{D}_{\beta(\theta)}} \qty[\ell(z; \theta) \nabla_\beta \ln p(z;\beta(\theta))]. \label{eq:nabla_L2}
\end{align}
In deriving \cref{eq:nabla_L2}, we used the so-called log-derivative trick, $\nabla_\beta p(z;\beta(\theta)) = p(z;\beta(\theta)) \nabla_\beta \ln p(z;\beta(\theta))$.
In \cref{eq:nabla_L}, the gradient is decomposed into two terms.
The first term $\nabla_1 \mathcal{L}(\theta)$ is the standard risk gradient, and the second term $\nabla_2 \mathcal{L}(\theta)$ is the indirect effect through the distribution shift.

\section{Proposed Algorithm}

We propose a projected stochastic gradient method for the performative prediction problem, described in \cref{alg:performative_prediction_1st_order}.
At iteration $t$ with the current model parameters $\theta_t$, we aim to update parameters along $-\nabla \mathcal{L}(\theta_t)$.
Using the decomposition in \cref{eq:nabla_L}, we estimate its two components from the corresponding samples.
Then, we update $\theta_t$ by projected gradient descent.
We detail each step of the algorithm in the following subsections.

\begin{algorithm}[t]
    \caption{Distribution-Shift-Aware Projected Gradient Method (DSA-PGM) for Performative Prediction}
    \label{alg:performative_prediction_1st_order}
    \begin{algorithmic}[1]
        \REQUIRE Initial model parameter $\theta_0 \in \Theta_\delta$, batch size $b \in \mathbb{N}$,
        perturbation radius $\delta > 0$, and stepsize $\alpha>0$.
        \FOR{$t = 0, 1, 2, \dots$}
        \STATE Deploy models and collect samples for $i=1,\dots,n$:\\
        $\{ z_{t,j} \}_{j=1}^{b} \sim \mathcal{D}_{\beta(\theta_t)}, \quad
            \{ z_{t,i,j}^{\pm} \}_{j=1}^{b} \sim \mathcal{D}_{\beta(\theta_t \pm \delta e_i)}$.
        \label{line:deploy_and_collect}

        \STATE Estimate distribution parameters and Jacobian: \\
        $
            \begin{aligned}
                 & \hat\beta_t \gets \Pred \left(\{z_{t,j}\}_{j=1}^{b} \right), \, \hat\beta_{t,i}^{\pm} \gets \Pred \left(\{z_{t,i,j}^{\pm}\}_{j=1}^{b} \right),                                                                        \\
                 & \widehat{\pdv{\beta}{\theta}}(\theta_t) \gets \mqty[\frac{\hat\beta_{t,1}^{+} - \hat\beta_{t,1}^{-}}{2\delta},                                 & \dots, & \frac{\hat\beta_{t,n}^{+} - \hat\beta_{t,n}^{-}}{2\delta}].
            \end{aligned}
        $
        \label{line:1st_estimate}

        \STATE Compute gradient components and overall gradient: \\
        $
            \begin{aligned}
                \widehat{\nabla_1\mathcal{L}}(\theta_t) & \gets \mathbb{E}_{z \sim \mathcal{D}_{\hat\beta_t}} \qty[\nabla_\theta \ell(z; \theta_t)], \\
                \widehat{\nabla_2\mathcal{L}}(\theta_t) & \gets \widehat{\pdv{\beta}{\theta}}(\theta_t)^\top G(\theta_t; \hat\beta_t),               \\
                \widehat{\nabla\mathcal{L}}(\theta_t)   & \gets \widehat{\nabla_1\mathcal{L}}(\theta_t) + \widehat{\nabla_2\mathcal{L}}(\theta_t).
            \end{aligned}
        $
        \label{line:1st_grad}

        \STATE Update model parameters with the projection: \\
        $
            \theta_{t+1} \gets \Proj_{\Theta_\delta}\left(\theta_t - \alpha \widehat{\nabla\mathcal{L}}(\theta_t)\right).
        $
        \label{line:1st_update}
        \ENDFOR
    \end{algorithmic}
\end{algorithm}

\paragraph{Estimation of $\nabla_1\mathcal{L}(\theta_t)$}

First, we estimate $\nabla_1\mathcal{L}(\theta_t)$, the standard risk gradient.
We collect data $\{z_{t,j}\}_{j=1}^{b}$ under model parameters $\theta_t$ and estimate the distribution parameter $\hat\beta_t$ by $\Pred$.
By \cref{eq:nabla_L1}, $\nabla_1\mathcal{L}(\theta_t)$ is approximated by $\widehat{\nabla_1\mathcal{L}}(\theta_t)=\mathbb{E}_{z\sim\mathcal{D}_{\hat\beta_t}}[\nabla_\theta\ell(z;\theta_t)]$.
Once $\hat\beta_t$ is estimated, this expectation can be evaluated using a closed-form expression or deterministic numerical integration without additional samples.
In the following analysis, we assume that $\widehat{\nabla_1\mathcal{L}}(\theta_t)$ is computed exactly or with negligible numerical error because we focus on sample complexity.

\paragraph{Estimation of $\nabla_2\mathcal{L}(\theta_t)$}

Next, we estimate $\nabla_2\mathcal{L}(\theta_t)$, the indirect effect through the distribution shift.
For $z \in \mathbb{R}^d$, $\beta_0 \in B$, and $\theta_t \in \Theta$, we define the score function $s(z;\beta_0)$ and the term $G(\theta_t; \beta_0)$ in \cref{eq:nabla_L2} as follows:
\begin{align}
    s(z;\beta_0)         & \coloneqq \nabla_{\beta} \ln p(z;\beta_0), \label{eq:score_function}                                              \\
    G(\theta_t; \beta_0) & \coloneqq \mathbb{E}_{z \sim \mathcal{D}_{\beta_0}} \qty[\ell(z; \theta_t) s(z;\beta_0)]. \label{eq:G_definition}
\end{align}
Using \cref{eq:G_definition}, we can rewrite $\nabla_2\mathcal{L}(\theta_t)$ in \cref{eq:nabla_L2} as
\begin{equation*}
    \nabla_2\mathcal{L}(\theta_t)   = {\pdv{\beta}{\theta}}(\theta_t)^\top G(\theta_t; \beta(\theta_t)).
\end{equation*}
Thus, we need to estimate the Jacobian ${\pdv{\beta}{\theta}}(\theta_t)$ and the expectation term $G(\theta_t; \beta(\theta_t))$.

To estimate the Jacobian ${\pdv{\beta}{\theta}}(\theta_t) \in \mathbb{R}^{m \times n}$, we use centered finite differences.
Let $e_i$ denote the $i$-th standard basis vector ($1 \leq i \leq n$) in the model parameter space $\mathbb{R}^n$.
For every $i \in \{1, \dots, n\}$, collect data $\{z_{t,i,j}^{\pm}\}_{j=1}^{b}$ under the perturbed model $\theta_t \pm \delta e_i$ where $\delta > 0$ is a small perturbation radius.
Then, we estimate the distribution parameter $\hat\beta_{t,i}^{\pm}$ for the perturbed model by $\Pred$ and compute the finite-difference approximation of the Jacobian:
\begin{equation*}
    \widehat{\pdv{\beta}{\theta}}(\theta_t) \gets \mqty[\frac{\hat\beta_{t,1}^+ - \hat\beta_{t,1}^-}{2\delta}, & \hdots, & \frac{\hat\beta_{t,n}^+ - \hat\beta_{t,n}^-}{2\delta}].
\end{equation*}
Although $\theta_t \pm \delta e_i$ may not belong to the feasible set $\Theta$, remedies are discussed in the ``Model Parameter Update.''

The term $G(\theta_t; \beta(\theta_t))$ can be estimated similarly to $\nabla_1\mathcal{L}(\theta_t)$.
The key difference is that we can utilize a so-called baseline \citep[Sections 2.8 and 13.4]{sutton2018reinforcement}.
For $\beta_0 \in B$, we have
\begin{equation}
    G(\theta_t; \beta_0)   = \mathbb{E}_{z \sim \mathcal{D}_{\beta_0}} \qty[(\ell(z; \theta_t) - \bar{\ell}(\theta_t;\beta_0)) s(z;\beta_0)],     \label{eq:G}
\end{equation}
where $\bar{\ell}$ is the expected loss:
\begin{equation}
    \bar{\ell}(\theta_t;\beta_0) \coloneqq \mathbb{E}_{z \sim \mathcal{D}_{\beta_0}}[\ell(z;\theta_t)]. \label{eq:baseline}
\end{equation}
This subtraction is justified since the expectation of $s(z;\beta_0)$ is zero under the appropriate assumptions.
By \cref{eq:score_function}, the log-derivative trick, and \cref{item:regularity} later specified, we have
\begin{equation*}
      \mathbb{E}_{z \sim \mathcal{D}_{\beta_0}}[s(z;\beta_0)] = \int p(z;\beta_0) s(z;\beta_0) \dd{\mu(z)}
    = \int \nabla_{\beta} p(z;\beta_0) \dd{\mu(z)} = \nabla_{\beta} \int p(z;\beta_0) \dd{\mu(z)} = 0.
\end{equation*}
Compared with no subtraction ($\bar{\ell}(\theta_t;\beta_0)=0$), this can reduce the estimator variance when the computation is inexact.

\paragraph{Model Parameter Update}

Finally, we update the model parameters with a standard projected gradient descent step using $\widehat{\nabla\mathcal{L}}(\theta_t) = \widehat{\nabla_1\mathcal{L}}(\theta_t) + \widehat{\nabla_2\mathcal{L}}(\theta_t)$.
For a rigorous treatment, we define the shrunk effective feasible set as
\begin{equation*}
    \Theta_\delta \coloneqq \left\{ \theta\in\Theta \relmiddle| \theta\pm\delta e_i \in \Theta,\ i = 1,\dots,n \right\}
\end{equation*}
and we also define the projection onto it as $\Proj(x) \coloneqq \argmin_{y \in \Theta_\delta} \norm{x - y}$.
When $\Theta_\delta$ is a nonempty closed convex set, $\Proj$ is uniquely defined for all $x \in \mathbb{R}^n$ \citep{valentine1964convex}.
When $\Theta = \Theta_\delta = \mathbb{R}^n$, we have $\Proj(x) = x$ for all $x \in \mathbb{R}^n$.
Then, with a constant stepsize $\alpha$, we can update the model parameters as $\theta_{t+1} \gets \Proj(\theta_t - \alpha \widehat{\nabla\mathcal{L}}(\theta_t))$.

We use $\Theta_\delta$ because estimating $\nabla_2\mathcal{L}$ requires deploying perturbed models with parameters $\theta_t \pm \delta e_i$, which may leave the feasible region.
By using $\Theta_\delta$ for $\Proj$, we can rigorously ensure $\theta_t \pm \delta e_i \in \Theta$ over all iterates.
In many applications, $\Theta$ is sufficiently large so that the iterates stay inside $\Theta$, or deployments are possible even outside $\Theta$.
Thus, we can also simply use $\Theta$ for $\Proj$.

\section{Theoretical Analysis}

In this section, we establish theoretical guarantees for \cref{alg:performative_prediction_1st_order}.
For $\beta_0 \in B$, we define $I(\beta_0)$ as the Fisher information matrix~\citep{malagoInformationGeometryGaussian2015,lyTutorialFisherInformation2017}:
\begin{equation}
    I(\beta_0) \coloneqq \mathbb{E}_{z \sim \mathcal{D}_{\beta_0}} \left[ s(z;\beta_0) s(z;\beta_0)^\top \right]. \label{eq:Fisher_information}
\end{equation}
We also define $\beta(\Theta) \coloneqq \{\beta(\theta)\}_{\theta \in \Theta}$.

\paragraph{Assumptions}

We state the assumptions used throughout the paper.
These assumptions are sufficiently general and consistent with the examples listed under ``Our Contributions'' in \cref{sec:intro}.
Details are provided in Appendix~\ref{app:examples_satisfying_assumptions}.
\begin{assumption}\label{asm:main}
    We assume the following conditions:
    \begin{assumptionitems}
        \item The model parameter space $\Theta$ and its shrinkage $\Theta_\delta$ are nonempty, closed, convex, and compact sets. \label{item:Theta_compact_convex}

        \item The distribution parameter space $B$ is a convex and compact set, and satisfies $\beta(\Theta)\subseteq B$. \label{item:B_compact_convex}

        \item The distribution parameter mapping $\beta$ is $C^3$. \label{item:beta_C3}

        \item The density or mass functions $\{p(\cdot;\beta_0)\}_{\beta_0 \in B}$ have common support, $p(z;\cdot)$ is $C^2$ on $B$, and differentiation under the integral sign with respect to $\beta_0$ is valid. \label{item:regularity}

        \item The norms of the score function $\norm{s(z;\beta_0)}$ and its derivative $\norm{\pdv{\beta} s(z;\beta_0)}$ have at most polynomial growth in $\norm{z}$, uniformly over $\beta_0\in B$. \label{item:poly_score}

        \item The loss function $\ell(z;\theta)$ is $C^2$. Moreover, $\abs{\ell(z;\theta)}$, $\norm{\nabla_\theta \ell(z;\theta)}$, and $\norm{\nabla_\theta^2 \ell(z;\theta)}$ have at most polynomial growth in $\norm{z}$, uniformly over $\theta\in\Theta$. \label{item:poly_ell}

        \item The family $\{\mathcal{D}_{\beta_0}\}_{\beta_0\in B}$ has uniformly upper-bounded finite moments up to order $q_{\max}$ defined in \cref{eq:q_max_definition}. \label{item:light_tailed}

        \item The distribution parameters admit a uniform estimation bound: there exist constants $\sigma_{\max} \geq 0$ and $b_{\min} \in \mathbb{N}$ such that, for any $\theta \in \Theta$ and any $b \geq b_{\min}$, the estimator $\hat\beta = \Pred(\{z_i\}_{i=1}^b) \in B$ satisfies
        \begin{equation*}
            \mathbb{E}\qty[\norm{\hat\beta - \beta(\theta)}^2] \le \frac{\sigma_{\max}^2}{b},
        \end{equation*}
        where $z_i$ are i.i.d.\ samples from $\mathcal{D}_{\beta(\theta)}$. \label{item:identifiability}
    \end{assumptionitems}
\end{assumption}

\Cref{asm:main} holds in a broad range of practical applications, going beyond the fixed-variance Gaussian location family and non-degenerate losses considered by existing gradient-based performative methods.
For \cref{item:Theta_compact_convex,item:B_compact_convex}, compact parameter spaces are often natural in practice, and one may further take $\Theta=\mathbb{R}^n$ and assume that the iterates $\{\theta_t\}$ remain in a sufficiently large compact subset, which is often the case in practice.
For \cref{item:beta_C3}, the condition means that the distribution shift is sufficiently smooth, which is natural in many applications.
For \cref{item:regularity,item:poly_score,item:poly_ell}, common parameterized families, score functions, and losses satisfy these conditions, including the previously mentioned examples.
For \cref{item:light_tailed}, even some heavy-tailed distributions satisfy the condition, such as the log-normal family with sufficiently well-behaved parameters.
Finally, for \cref{item:identifiability}, the $\order{b^{-1}}$ estimation rate yields the sample complexity stated below and facilitates comparison with existing methods.
A slower and consistent estimator would worsen the sample complexity but would not preclude a convergence guarantee.

Under \cref{asm:main}, we can derive the following technical conditions.
The proof is provided in Appendix~\ref{app:proof_prop1}. 
\newcommand{\propVerification}{%
    Under \cref{asm:main}, \cref{asm:PL_bounded,asm:ell_theta,asm:beta_theta,asm:G_Lipschitz,asm:ell_bounded,asm:Fisher} hold.%
}
\begin{proposition}\label{prop:technical_conditions}
    \propVerification
\end{proposition}

\begin{condition}\label{asm:PL_bounded}
    The performative loss $\mathcal{L}$ is bounded below:
    \begin{equation*}
        \mathcal{L}_* \coloneqq \inf_{\theta \in \Theta} \mathcal{L}(\theta) > -\infty.
    \end{equation*}
\end{condition}
\begin{condition}\label{asm:ell_theta}
    The loss function $\ell$ is $L_{\ell,\theta}^{\mathrm{Lip}}$-Lipschitz continuous and $L_{\ell,\theta}^{\mathrm{sm}}$-smooth:
    \begin{align*}
        \sup_{\beta \in B} \mathbb{E}_{z \sim \mathcal{D}_{\beta}} \qty[\norm{\nabla_{\theta} \ell(z; \theta)}^2]^{1/2}                & \le L_{\ell,\theta}^{\mathrm{Lip}},                        \\
        \mathbb{E}_{z \sim \mathcal{D}_{\beta(\theta)}}\qty[\norm{\nabla_{\theta} \ell(z; \theta) - \nabla_{\theta} \ell(z; \theta')}] & \le L_{\ell,\theta}^{\mathrm{sm}} \norm{\theta - \theta'}.
    \end{align*}
\end{condition}
\begin{condition}\label{asm:beta_theta}
    The mapping $\beta$ is $L_{\beta,\theta}^{\mathrm{Lip}}$-Lipschitz continuous, $L_{\beta,\theta}^{\mathrm{sm}}$-smooth, and $L_{\beta,\theta}^{\mathrm{H}}$-Hessian Lipschitz continuous:
    \begin{align*}
        \norm{\beta(\theta) - \beta(\theta')}                                 & \le L_{\beta,\theta}^{\mathrm{Lip}} \norm{\theta - \theta'}, \\
        \norm{{\pdv{\beta}{\theta}}(\theta) - {\pdv{\beta}{\theta}}(\theta')} & \le L_{\beta,\theta}^{\mathrm{sm}} \norm{\theta - \theta'},  \\
        \norm{{\pdv[2]{\beta}{\theta}}(\theta) - {\pdv[2]{\beta}{\theta}}(\theta')}
                                                                              & \le L_{\beta,\theta}^{\mathrm{H}} \norm{\theta - \theta'}.
    \end{align*}
\end{condition}
\begin{condition}\label{asm:G_Lipschitz}
    The gradient component function $G$ is jointly $(L_{G,\beta}^{\mathrm{Lip}},L_{G,\theta}^{\mathrm{Lip}})$-Lipschitz continuous:
    \begin{equation*}
        \norm{G(\theta;\beta_0)-G(\theta';\beta_0')}
        \le
        L_{G,\beta}^{\mathrm{Lip}} \norm{\beta_0-\beta_0'}
        +
        L_{G,\theta}^{\mathrm{Lip}} \norm{\theta-\theta'}.
    \end{equation*}
\end{condition}
\begin{condition}\label{asm:ell_bounded}
    The loss variance is uniformly bounded by a constant $M_\ell \geq 0$ for all $\theta \in \Theta$ and $\beta_0 \in B$:
    \begin{equation*}
        \mathrm{Var}_{z \sim \mathcal{D}_{\beta_0}} \left[ \ell(z; \theta) \right] \leq M_\ell.
    \end{equation*}
\end{condition}
\begin{condition}\label{asm:Fisher}
    The trace of the Fisher information matrix is uniformly bounded by a constant $I_{\max} \geq 0$ for all $\beta_0 \in B$:
    \begin{equation*}
        \Tr(I(\beta_0)) \le I_{\max}.
    \end{equation*}
\end{condition}

\paragraph{Technical Propositions}

Now, we derive the convergence guarantee for the proposed method.
The first key result is the $\Lsm$-smoothness of the performative loss.
\newcommand{\propPLSmoothStatement}[1][align]{%
    Suppose that \cref{asm:main} holds.
    Then, the performative loss $\mathcal{L}$ is $\Lsm$-smooth; for any $\theta,\theta' \in \Theta$,
    \begin{equation*}
        \norm{\nabla \mathcal{L}(\theta)-\nabla \mathcal{L}(\theta')}
        \le
        \Lsm \norm{\theta-\theta'},
    \end{equation*}
    where
    \ifthenelse{\equal{#1}{equation}}{%
        \begin{equation*}
            \Lsm
            \coloneqq
            L_{\ell,\theta}^{\mathrm{sm}} + L_{\ell,\theta}^{\mathrm{Lip}} L_{\beta,\theta}^{\mathrm{Lip}} \sqrt{I_{\max}}
            + L_{\beta,\theta}^{\mathrm{sm}} \sqrt{M_\ell I_{\max}} + L_{\beta,\theta}^{\mathrm{Lip}} \qty(L_{G,\beta}^{\mathrm{Lip}} L_{\beta,\theta}^{\mathrm{Lip}}+L_{G,\theta}^{\mathrm{Lip}}).
        \end{equation*}
    }{%
        \begin{align*}
            \Lsm
            \coloneqq &
            L_{\ell,\theta}^{\mathrm{sm}} + L_{\ell,\theta}^{\mathrm{Lip}} L_{\beta,\theta}^{\mathrm{Lip}} \sqrt{I_{\max}}                                                                                         \\
                      & \, + L_{\beta,\theta}^{\mathrm{sm}} \sqrt{M_\ell I_{\max}} + L_{\beta,\theta}^{\mathrm{Lip}} \qty(L_{G,\beta}^{\mathrm{Lip}} L_{\beta,\theta}^{\mathrm{Lip}}+L_{G,\theta}^{\mathrm{Lip}}).
        \end{align*}
    }
}
\begin{proposition}\label{prop:PL_LSmooth}
    \propPLSmoothStatement[equation]
\end{proposition}

The proof is provided in Appendix~\ref{app:proof_prop2}.
This smoothness result permits regular parametric families under the required conditions, extending the existing result \citep[Lemma~1]{rayDecisionDependentRiskMinimization2022}, which assumes a location--scale distribution family and globally Lipschitz derivatives of the loss.

Next, for stepsize $\alpha > 0$ and parameter $\theta$, we define the gradient mapping $\mathcal{G}_{\alpha}(\theta)$, which generalizes the gradient to constrained optimization problems \citep{beckChapter10Proximal2017}:
\begin{equation*}
    \begin{dcases}
        \mathcal{G}_{\alpha}(\theta)
         & \coloneqq
        \frac{1}{\alpha}
        \qty(\theta - \Proj\qty(\theta - \alpha \nabla \mathcal{L}(\theta))), \\
        \widehat{\mathcal{G}_{\alpha}}(\theta)
         & \coloneqq
        \frac{1}{\alpha}
        \qty(\theta - \Proj\qty(\theta - \alpha \widehat{\nabla\mathcal{L}}(\theta))).
    \end{dcases}
\end{equation*}
By definition, $\mathcal{G}_{\alpha}(\theta) = \nabla \mathcal{L}(\theta)$ if $\Theta = \mathbb{R}^n$, and
\begin{equation}
    \theta_{t+1}  = \Proj\qty(\theta_t - \alpha \widehat{\nabla\mathcal{L}}(\theta_t)) = \theta_t - \alpha \widehat{\mathcal{G}_{\alpha}}(\theta_t), \label{eq:gradient_mapping_update}
\end{equation}
which naturally generalizes the standard gradient descent method $\theta_{t+1} = \theta_t - \alpha \widehat{\nabla \mathcal{L}}(\theta_t)$.

The next proposition gives a general bound on the average squared norm of the gradient mapping.

\newcommand{\propBoundGradientNormStatement}[1][align]{%
    Suppose that \cref{asm:main} holds.
    Define $\{\theta_t\}_{t=0}^{T}$ as the sequence generated by the projected gradient descent $\theta_{t+1}=\Proj\qty(\theta_t-\alpha\widehat{\nabla\mathcal{L}}(\theta_t))$, initialized at $\theta_0 \in \Theta_\delta$ with a stepsize $0 < \alpha \leq 1/\Lsm$.
    Then, the mean squared gradient mapping up to iteration $T$ is bounded as
    \ifthenelse{\equal{#1}{equation}}{%
        \begin{equation*}
            \frac{1}{T} \sum_{t=0}^{T-1} \mathbb{E} \left[\norm{\mathcal{G}_{\alpha}(\theta_t)}^2 \right]
            \le \frac{2 \qty(\mathcal{L}(\theta_0)-\mathcal{L}_*)}{\alpha T}+\frac{1}{T}\sum_{t=0}^{T-1}\mathbb{E}\left[\norm{\nabla\mathcal{L}(\theta_t) - \widehat{\nabla\mathcal{L}}(\theta_t)}^2\right].
        \end{equation*}
    }{%
        \begin{align*}
                  & \frac{1}{T} \sum_{t=0}^{T-1} \mathbb{E} \left[\norm{\mathcal{G}_{\alpha}(\theta_t)}^2 \right]                                                                                                \\
            \le{} & \frac{2 \qty(\mathcal{L}(\theta_0)-\mathcal{L}_*)}{\alpha T}+\frac{1}{T}\sum_{t=0}^{T-1}\mathbb{E}\left[\norm{\nabla\mathcal{L}(\theta_t) - \widehat{\nabla\mathcal{L}}(\theta_t)}^2\right].
        \end{align*}
    }
}
\begin{proposition}\label{prop:bound_gradient_norm}
    \propBoundGradientNormStatement[equation]
\end{proposition}
The proof is provided in Appendix~\ref{app:proof_prop3}.
The first term in this bound vanishes as the number of iterations $T$ increases.
Therefore, controlling the discrepancy between the true and estimated gradients in the second term is the key to guaranteeing convergence.

We bound the summand of this second term primarily using \cref{item:identifiability}.
The proof is provided in Appendix~\ref{app:proof_prop4}.

\newcommand{\propBoundLOneLTwo}[1][align]{%
    Suppose that \cref{asm:main} holds.
    Let the perturbation radius $\delta$ be $b^{-1/6}$ with $b \geq b_{\min}$.
    Then, for any iteration $t$, the estimated gradient $\widehat{\nabla\mathcal{L}}(\theta_t)$ in \cref{alg:performative_prediction_1st_order} satisfies the following error bound:
    \begin{equation*}
        \mathbb{E}\left[
            \norm{\nabla\mathcal{L}(\theta_t) - \widehat{\nabla\mathcal{L}}(\theta_t)}^2
            \right]
        \leq{}   \frac{C_1}{b} + \frac{n C_2}{b^{2/3}},
    \end{equation*}
    where
    \ifthenelse{\equal{#1}{equation}}{%
        \begin{equation*}
            C_1 \coloneqq 3\qty(\qty(L_{\ell,\theta}^{\mathrm{Lip}})^2 I_{\max}
            + \qty(L_{\beta,\theta}^{\mathrm{Lip}}L_{G,\beta}^{\mathrm{Lip}})^2)\sigma_{\max}^2, \qquad
            C_2 \coloneqq
            3 M_\ell I_{\max}
            \qty(
            \frac{(L_{\beta,\theta}^{\mathrm{H}})^2}{12}
            +
            \frac{3\sigma_{\max}^2}{2}
            ).
        \end{equation*}
    }{%
        \begin{align*}
            C_1 & \coloneqq 3\qty(\qty(L_{\ell,\theta}^{\mathrm{Lip}})^2 I_{\max}
            + \qty(L_{\beta,\theta}^{\mathrm{Lip}}L_{G,\beta}^{\mathrm{Lip}})^2)\sigma_{\max}^2, \\
            C_2 & \coloneqq
            3 M_\ell I_{\max}
            \qty(
            \frac{(L_{\beta,\theta}^{\mathrm{H}})^2}{12}
            +
            \frac{3\sigma_{\max}^2}{2}
            ).
        \end{align*}
    }
}
\begin{proposition}\label{prop:bound_L1_L2}
    \propBoundLOneLTwo[equation]
\end{proposition}

\paragraph{Sample Complexity Analysis}

Combining \cref{prop:bound_gradient_norm,prop:bound_L1_L2}, we obtain the following convergence guarantee and sample complexity for \cref{alg:performative_prediction_1st_order}.
\begin{theorem}\label{thm:main}
    Suppose that \cref{asm:main} holds, and let $C_1$ and $C_2$ be defined as in \cref{prop:bound_L1_L2}.
    Consider \cref{alg:performative_prediction_1st_order} initialized at $\theta_0 \in \Theta_\delta$ with stepsize $0 < \alpha \leq 1/\Lsm$.
    For any $\epsilon > 0$, choose $T$ and $b$ as the smallest integers satisfying
    \begin{equation*}
        T \geq \frac{6\bigl(\mathcal{L}(\theta_0)-\mathcal{L}_*\bigr)} {\alpha\epsilon^2},
        \:
        b \geq \max\qty(
        b_{\min},
        \frac{3C_1}{\epsilon^2},
        \frac{(3 n C_2)^{1.5}}{\epsilon^3}
        ),
    \end{equation*}
    and set perturbation radius $\delta = b^{-1/6}$.
    Then we have
    \begin{equation*}
        \frac{1}{T} \sum_{t=0}^{T-1} \mathbb{E} \left[ \norm{\mathcal{G}_{\alpha}(\theta_t)}^2 \right] \leq \epsilon^2,
    \end{equation*}
    and the total sample complexity of the algorithm is
    \begin{equation*}
        T \cdot (2n+1) b = \order{n^{2.5} \epsilon^{-5}}.
    \end{equation*}
\end{theorem}
\begin{proof}
    By combining \cref{prop:bound_gradient_norm,prop:bound_L1_L2}, we have
    \begin{equation*}
        \frac{1}{T} \sum_{t=0}^{T-1} \mathbb{E} \left[\norm{\mathcal{G}_{\alpha}(\theta_t)}^2 \right]
        \leq
        \frac{2 \qty(\mathcal{L}(\theta_0)-\mathcal{L}_*)}{\alpha T}
        + \frac{C_1}{b}
        + \frac{n C_2}{b^{2/3}}.
    \end{equation*}
    When the conditions on $T$ and $b$ hold, each term on the right-hand side is bounded by $\epsilon^2/3$, concluding the proof.
\end{proof}

\begin{table*}[t]
    \centering
    \caption{Comparison of Methods.}
    \label{tab:methods}
    {\scriptsize
        \setlength{\tabcolsep}{2pt} 
        \renewcommand{\arraystretch}{0.95} 
        \begin{tabular}{
                C{0.215\linewidth}
                C{0.24\linewidth}
                C{0.135\linewidth}
                C{0.37\linewidth}
            }
            \toprule
            Method                                                                          & Estimation/Update                                                                               & Samples Per Iter. & Description                                                    \\
            \midrule
            \textbf{RRM} \scalebox{0.8}{\citep{perdomoPerformativePrediction2020}}          & $\theta_{t+1} \gets \argmin_{\theta}\sum_{j=1}^{b}\ell(z_{t,j};\theta)$                         & $b$               & Retrain without considering distribution shift                 \\
            \textbf{RGD} \scalebox{0.8}{\citep{perdomoPerformativePrediction2020}}          & $g_t \gets \widehat{\nabla_1 \mathcal{L}}(\theta_t)$                                            & $b$               & Ignore distribution shift                                      \\
            \textbf{Plug-in} \scalebox{0.8}{\citep{linPluginPerformativeOptimization2024b}} & $\hat{\theta} \gets \argmin_{\theta}\,\hat{\mathcal{L}}(\theta)$                                & pre-sampled       & Requires knowledge of $\beta(\theta)$                          \\
            \textbf{DFO} \scalebox{0.75}{\citep{hikimaZerothorderGradientEstimators2025}}   & $g_t \gets \frac{n}{\delta}\,\mathbb{E}_{u} \left[\mathcal{L}(\theta_t+\delta u)u\right]$       & $b$               & Coordinate-direction zeroth-order method                       \\
            \textbf{PerfGD} \scalebox{0.8}{\citep{izzoHowLearnWhen2021d}}                   & $g_t \gets \widehat{\nabla_1 \mathcal{L}}(\theta_t) + \widehat{\nabla_2 \mathcal{L}}(\theta_t)$ & $b$               & Estimate $\nabla_2 \mathcal{L}$ with pseudo-inverse of history \\
            \textbf{Proposed}                                                               & $g_t \gets \widehat{\nabla_1 \mathcal{L}}(\theta_t) + \widehat{\nabla_2 \mathcal{L}}(\theta_t)$ & $(2n+1)b$         & Estimate $\nabla_2 \mathcal{L}$ with coordinate perturbations  \\
            \textbf{Proposed (cyclic)}                                                      & $g_t \gets \widehat{\nabla_1 \mathcal{L}}(\theta_t) + \widehat{\nabla_2 \mathcal{L}}(\theta_t)$ &
            \makecell{\scalebox{0.9}{initially} $(2n+1)b$;                                                                                                                                                                                                                         \\ \scalebox{0.9}{thereafter} $b$ \scalebox{0.9}{or} $3b$}
                                                                                            & Cyclically refresh one column of $\widehat{\pdv{\beta}{\theta}}$                                                                                                                     \\
            \bottomrule
        \end{tabular}%
    }
\end{table*}

A slight modification of the assumptions, combined with smoothing on the sphere \citep{hikimaZerothorderGradientEstimators2025}, improves the sample complexity.
The formal statement and proof of \cref{prop:sphere_smoothing} are provided in Appendix~\ref{app:sphere_smoothing}.

\begin{proposition}[informal]\label{prop:sphere_smoothing}
    Under a slight strengthening of \cref{asm:main}, \cref{alg:performative_prediction_1st_order} with a sphere-smoothed random-direction estimator improves the total sample complexity from $\order{n^{2.5} \epsilon^{-5}}$ to $\order{n^2 \epsilon^{-5}}$.
\end{proposition}

\section{Variant of the Proposed Method}

Having established convergence guarantees for \cref{alg:performative_prediction_1st_order}, we discuss a practical variant that reduces the number of additional deployments: ``Proposed (cyclic)''.
This exploits the fact that the Jacobian often changes gradually during optimization.
Specifically, this variant method first estimates all columns of $\widehat{\pdv{\beta}{\theta}}$ using the same coordinate-wise finite-difference procedure as \cref{alg:performative_prediction_1st_order}, and stores the resulting matrix in memory.
Thereafter, for a prescribed update interval $K\in\mathbb{N}$, one column is refreshed every $K$ iterations in cyclic order, while the remaining columns are reused.
As a result, each iteration requires $b$ samples when no column is refreshed and $3b$ samples when one column is refreshed, which are reduced from the original $(2n+1)b$ samples.
Although this approach is heuristic in this paper, its analysis may benefit from techniques developed for cyclic block coordinate methods \citep{caiCyclicBlockCoordinate2023,beckConvergenceBlockCoordinate2013,nesterovEfficiencyCoordinateDescent2012}.

We also used Adam \citep{kingmaAdamMethodStochastic2017a}
as an adaptive-stepsize variant in Appendix~\ref{app:stepsize_schedule}.
We conducted experiments with this variant, which showed similar trends to the constant-stepsize results and improved performance in practice.

\section{Experiments}
\label{sec:experiments}

In this section, we present the experimental results of the proposed methods on various synthetic problems, comparing them with existing approaches.

\paragraph{Comparison of Methods}

The methods compared in our experiments are summarized in \cref{tab:methods} from the perspective of their update rules.
RRM and RGD are the originally studied algorithms \citep{perdomoPerformativePrediction2020}.
RRM retrains the model at iteration $t$ on the batch $\{z_{t,j}\}_{j=1}^b$ collected under the current model parameter $\theta_t$:
\begin{equation*}
    \theta_{t+1} \gets \argmin_{\theta \in \Theta}
    \sum_{j=1}^{b}\ell(z_{t,j};\theta).
\end{equation*}
RGD ignores the effect of distribution shift and uses only the direct gradient term:
\begin{equation*}
    \widehat{\nabla \mathcal{L}}(\theta_t) \gets \widehat{\nabla_1 \mathcal{L}}(\theta_t).
\end{equation*}
The plug-in method \citep{linPluginPerformativeOptimization2024b} first estimates how the data distribution depends on the deployed model using samples collected across models.
For all experiments, we fit the Gaussian surrogate $z \sim \mathcal{N}(M_0+M_1\theta,\sigma^2 I_d)$ by linear regression, which provides a simple and uniform implementation across problem settings.
It then minimizes the induced estimated performative loss:
\begin{equation*}
    \hat{\theta}
    \gets
    \argmin_{\theta \in \Theta}\hat{\mathcal{L}}(\theta),
    \;\;
    \hat{\mathcal{L}}(\theta)
    \gets
    \mathbb{E}_{z \sim \mathcal{N}(\hat M_0+\hat M_1\theta, \hat\sigma^2 I_d)}[\ell(z;\theta)].
\end{equation*}
DFO \citep{hikimaZerothorderGradientEstimators2025} estimates the performative gradient from function evaluations rather than first-order information.
For a signed coordinate direction $u\sim\mathrm{Unif}(\{\pm e_i\}_{i=1}^{n})$, it uses the zeroth-order estimator:
\begin{equation*}
    \widehat{\nabla \mathcal{L}}(\theta_t)
    \gets
    \frac{n}{\delta}
    \mathbb{E}_{u}\!\left[
        \mathcal{L}(\theta_t+\delta u)u
        \right].
\end{equation*}
PerfGD \citep{izzoHowLearnWhen2021d} estimates the distributional component of the performative gradient from historical information using a pseudo-inverse.

We report the constant-stepsize results in the main text.
We also conducted hyperparameter tuning for each method.
See Appendix~\ref{app:experimental_setup} for further details.

\begin{figure*}[t]
    \centering
    \begin{minipage}{0.75\textwidth}
        \begin{tabular}{@{}cc@{}}
            \includegraphics[width=0.48\linewidth]{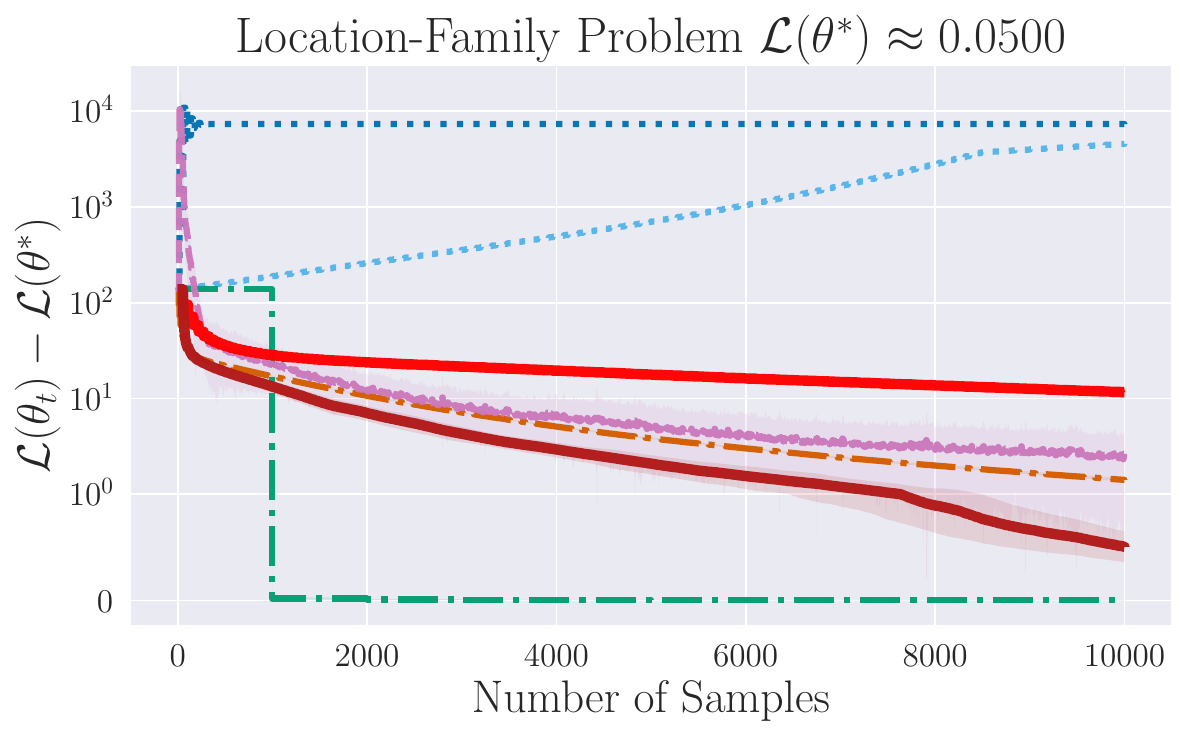}  &
            \includegraphics[width=0.48\linewidth]{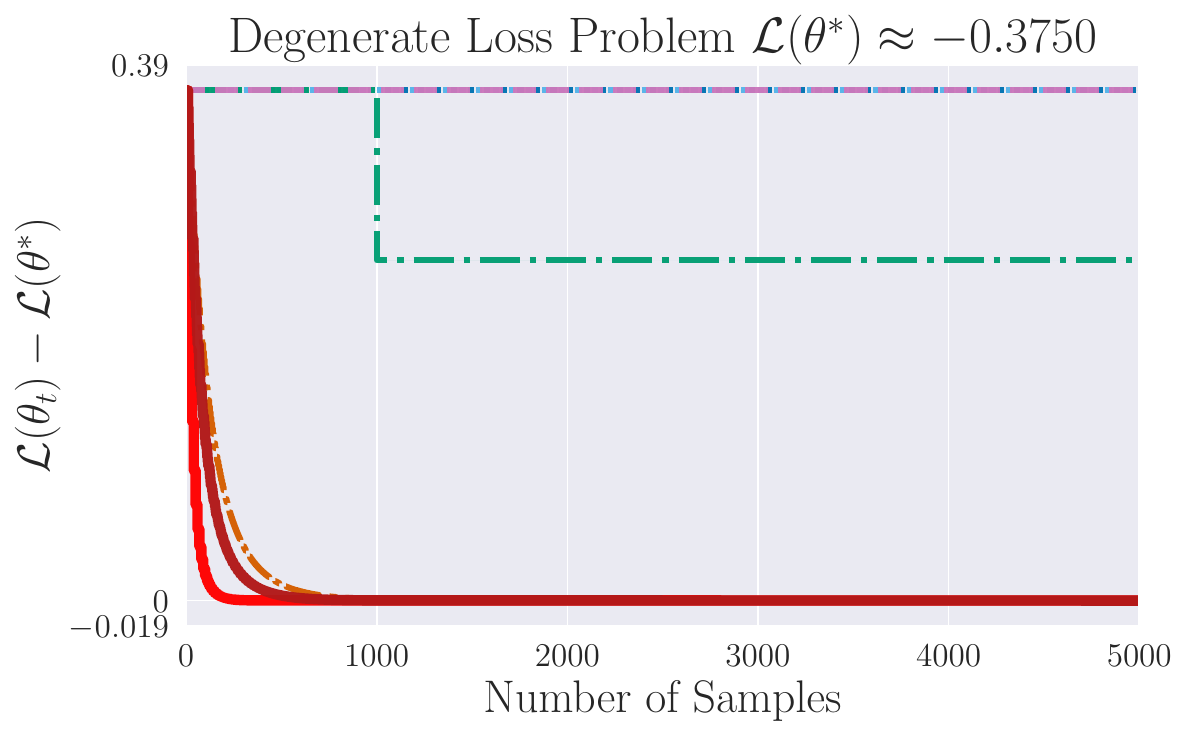} \\
            \includegraphics[width=0.48\linewidth]{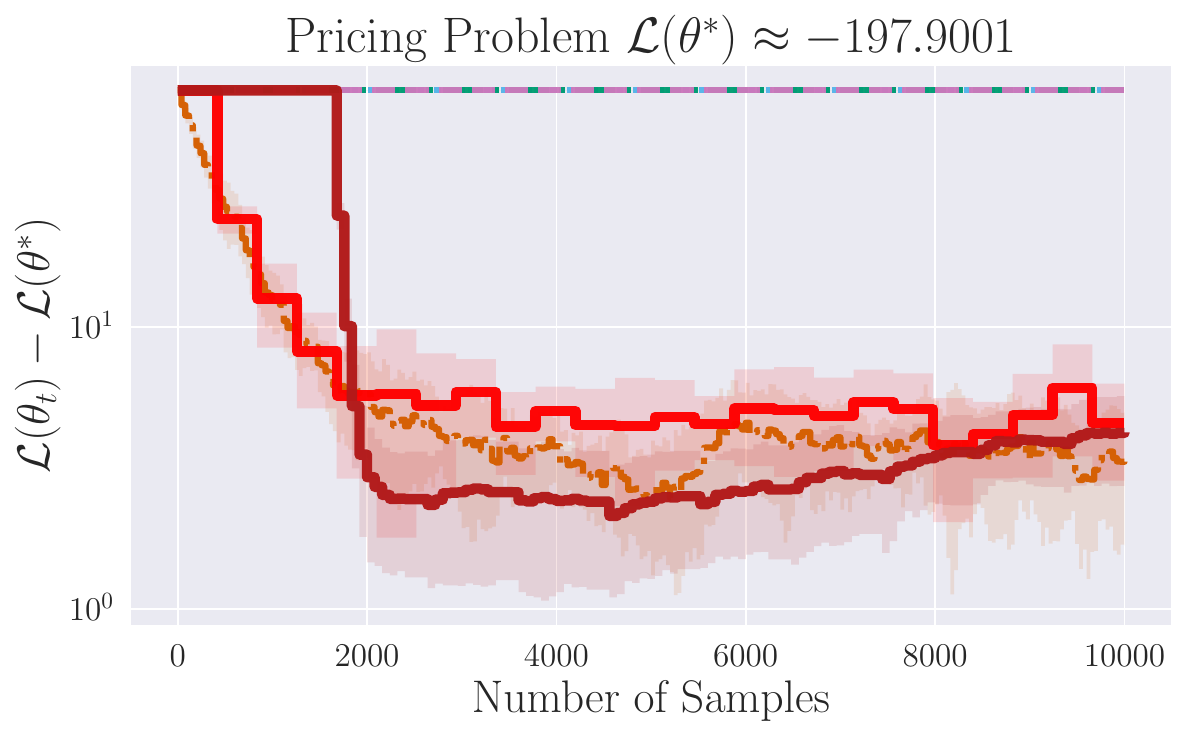} &
            \includegraphics[width=0.48\linewidth]{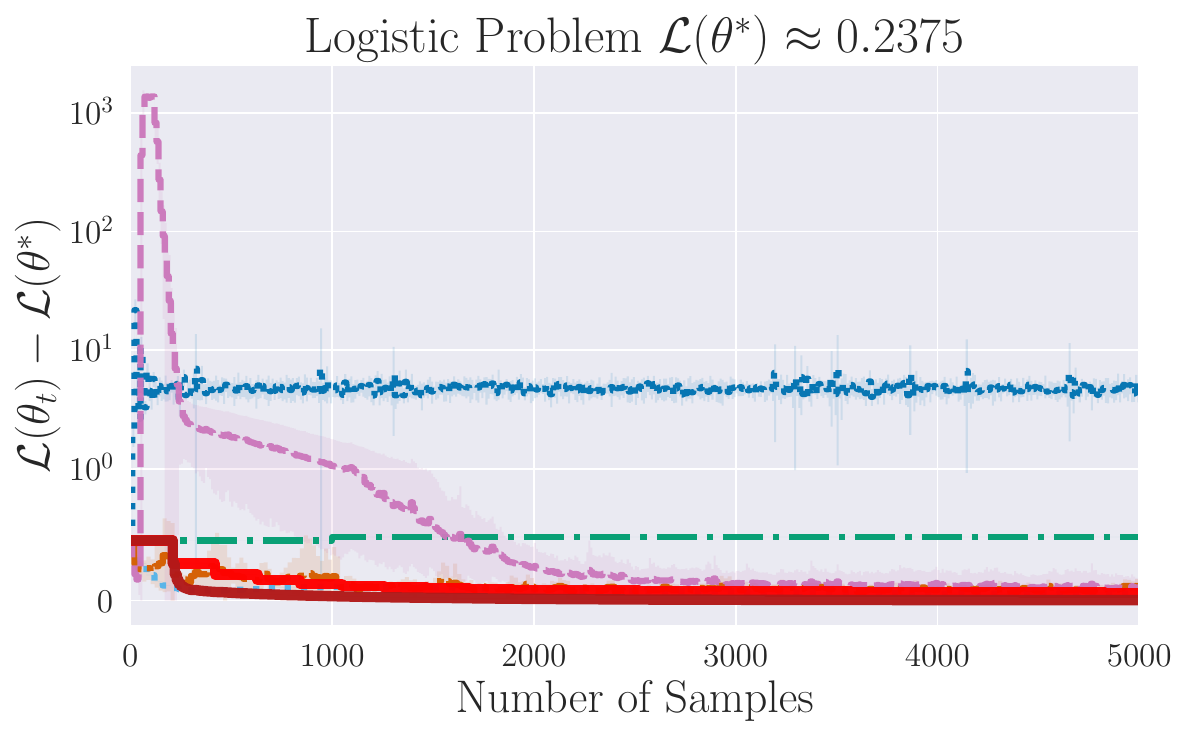}
        \end{tabular}
    \end{minipage}
    \hfill
    \begin{minipage}{0.2\textwidth}
        \includegraphics[width=\linewidth]{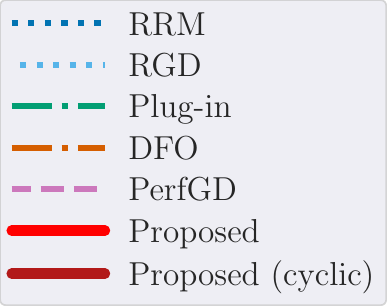}
    \end{minipage}
    \caption{Experimental results with constant stepsizes.
        For each problem, let $\theta^*$ denote an optimal or numerically best model parameter, and we evaluate each iterate using the excess performative risk $\mathcal{L}(\theta_t) - \mathcal{L}(\theta^*)$.
        Solid curves show the mean excess performative risk, and shaded regions indicate one standard deviation across 10 independent trials.
        The vertical axis uses a symmetric logarithmic scale, which is linear for values between $-1$ and $1$ and logarithmic outside this range.
        In the degenerate loss problem, the curves for RRM, RGD, and PerfGD nearly overlap near their initial objective values.
        Both proposed methods (red curves) exhibit stable convergence across all experiments.}
    \label{fig:experiments_result}
\end{figure*}

\paragraph{Experimental Setup}

We conducted experiments on several synthetic problem settings to evaluate the performance of the proposed methods.
The location-family problem is a simple Gaussian mean-shift model, $z=M_0+M_1\theta+\sigma\xi$ with $\xi\sim\mathcal{N}(0,I_d)$, in which the distribution parameter mapping $\beta$ is linear.
This is a standard setting in the literature \citep{perdomoPerformativePrediction2020,izzoHowLearnWhen2021d,linPluginPerformativeOptimization2024b}.
The degenerate loss problem is one in which the standard risk gradient $\nabla_1 \mathcal{L}(\theta) = \mathbb{E}[\nabla_\theta\ell(z;\theta)]$ can vanish even though the performative gradient $\nabla\mathcal{L}(\theta)$ is nonzero.
The pricing and logistic problems are more application-oriented settings in which the deployed decision changes demand or the positive-class feature distribution.
Details of the problem settings are provided in Appendix~\ref{app:problem_setup}.

\paragraph{Results}

We present the results in \cref{fig:experiments_result}.
The proposed methods perform favorably across all four settings, finding effective descent directions even when other methods fail.
See also the Wilcoxon signed-rank test results in Appendix~\ref{app:significance_tests}.

In the location-family problem, the distribution parameter mapping $\beta(\theta)$ is linear, and the problem is relatively simple.
The plug-in method can almost perfectly recover this mapping from sampled pairs $(\theta,z)$, and thus it is the fastest in this instance.
Since the proposed methods are designed for a broader class of problems, their convergence is relatively slow in this simple setting.
Still, both proposed methods converge, and Proposed (cyclic) is the second-fastest method.

In the degenerate loss problem, RRM, RGD, and PerfGD remain near their initial objective values, while the plug-in method improves the objective but then stalls due to the data distribution misspecification.
In contrast, both proposed methods and DFO converge to near-optimal objective values.
This behavior highlights one of the main contributions of our work from a practical perspective.
The proposed gradient-based methods extend the applicability of gradient-based performative optimization to these harder problem instances.
This can be understood directly from the structure of the problem.
At the initialization $\theta_0=(0,0)$, the standard risk gradient $\nabla_1 \mathcal{L}(\theta_0) = \mathbb{E}[\nabla_\theta\ell(z;\theta_0)]$ is zero, while the true performative gradient $\nabla\mathcal{L}(\theta_0)$ is nonzero.
PerfGD first relies on historical movement to estimate the distributional component, but if the initial loss-gradient update is zero, it does not generate informative nearby history and remains stuck.
By contrast, the proposed methods actively sample perturbed deployments around the current $\theta$, so they can estimate the missing distributional component and identify a direction that decreases the performative loss.

In the pricing and logistic problems, both Proposed and Proposed (cyclic) quickly approach the optimal solutions.
These stable and rapidly converging behaviors highlight the practicality of the proposed methods in these settings.

\section{Conclusion}

In this paper, we developed a gradient-based optimization framework for performative prediction.
The proposed methods have sample-complexity guarantees under broad classes of loss functions and distributions, and perform favorably across the numerical experiments.

Several directions remain for future work.
First, a natural direction is to extend our analysis to stateful settings \citep{izzoHowLearnWhen2022}.
Second, combining with direct estimation of the induced distribution shift may lead to more efficient algorithms.
The proposed framework may be inefficient in over-parameterized settings, where estimating $\beta(\theta)$ can require more samples than are necessary to characterize the relevant distribution shift.
Third, the methods rely on estimating the distribution parameters at $2n$ perturbed deployments, $\{\beta(\theta\pm\delta e_i)\}_{i=1}^n$.
As a result, gradient estimates can be sensitive to estimation error and numerical noise.
This issue is pronounced when the distribution shift is small, a regime outside the main focus of this paper.
Improving this behavior is another promising direction for future work.

\section{Acknowledgements}

This work was supported by JSPS KAKENHI Grant Numbers JP26KJ0936, JP23K28041.

\bibliographystyle{plainnat}
\bibliography{NTT2026.bib}


\newpage
\onecolumn
\appendix

\setcounter{proposition}{0}
\setcounter{theorem}{0}

\section{Details of the Assumptions}

In this section, we clarify the assumptions made in \cref{asm:main} and verify that there exist instances satisfying \cref{asm:main}.

\subsection{Exact Formulation of the Assumptions}
\label{app:exact_formulation_of_assumptions}

We first clarify the formal definition of $p$.
Let $\mathcal{Z} \subseteq \mathbb{R}^d$ be the data space, and let $\mathcal{A}$ be a $\sigma$-algebra on $\mathcal{Z}$.
Let $\{\mathcal{D}_{\beta_0}\}_{\beta_0 \in B}$ be dominated by a common $\sigma$-finite measure $\mu$ on $(\mathcal{Z}, \mathcal{A})$.
Then, we can define $p(z;\beta_0) = \dv{\mathcal{D}_{\beta_0}}{\mu}(z)$ as the Radon--Nikodym derivative.
This definition is compatible with \cref{item:regularity}.

We next state explicitly the polynomial-growth requirements in \cref{item:poly_score,item:poly_ell} and the moment requirement in \cref{item:light_tailed}.
We assume that there exist nonnegative finite constants $K_1, \dots, K_5$ and nonnegative integers $q_1, \ldots, q_5$ such that, for all $\theta \in \Theta$, $\beta_0 \in B$, and $z \in \mathbb{R}^d$,
\begin{align*}
    \norm{s(z;\beta_0)}
     & \le K_1 (1 + \norm{z}^{q_1}), \\
    \norm{\pdv{\beta} s(z;\beta_0)}
     & \le K_2 (1 + \norm{z}^{q_2}), \\
    \abs{\ell(z;\theta)}
     & \le K_3 (1 + \norm{z}^{q_3}), \\
    \norm{\nabla_{\theta}\ell(z;\theta)}
     & \le K_4 (1 + \norm{z}^{q_4}), \\
    \norm{\nabla^2_{\theta}\ell(z;\theta)}
     & \le K_5 (1 + \norm{z}^{q_5}).
\end{align*}
Let
\begin{equation}
    q_{\max} \coloneqq \max\{2 q_1 + q_3, q_2 + q_3, 2 q_3, 2 q_4, q_1 + q_4, q_5\}. \label{eq:q_max_definition}
\end{equation}
Then, \cref{item:light_tailed} implies there exist constants $\{M_k\}_{0 \le k \le q_{\max}}$ such that
\begin{equation}
    \sup_{\beta_0 \in B} \mathbb{E}_{z \sim \mathcal{D}_{\beta_0}} [\norm{z}^k] \le M_k < \infty. \label{eq:uniform_moment_bound}
\end{equation}
Note that $q_{\max}$ is large enough to control the following products of polynomial-growth quantities:
$\ell$,
$\ell^2$,
$\norm{\nabla_\theta \ell}^2$,
$\norm{s (\nabla_\theta \ell)^\top}$,
$\norm{\nabla^2_{\theta} \ell}$,
$\norm{\ell(s s^\top + \pdv{\beta} s)}$,
and $\norm{s}^2$.

\subsection{Examples Satisfying the Assumptions}
\label{app:examples_satisfying_assumptions}

We next discuss the standard examples satisfying \cref{asm:main}.

As for the loss function, \cref{item:poly_ell} is satisfied by many commonly used loss functions in finite-dimensional parametric models when $\Theta$ is compact (\cref{item:Theta_compact_convex}).
Examples include polynomial losses, logistic loss, and softmax cross-entropy loss.
With minor modifications to the analysis, the $C^2$ assumption of $\ell$ can be weakened to $C^1$ regularity together with a Lipschitz bound on $\nabla_\theta \ell(z;\theta)$ that grows at most polynomially in $\norm{z}$.
Consequently, non-$C^2$ losses such as smooth $L_1$ loss can be handled in essentially the same way.
See also \citet{tervenComprehensiveSurveyLoss2025,elharroussTaskbasedLossFunctions2025a}.

For the data distribution, its assumptions are likewise standard for regular parametric models.
We first check the moment condition stated as \cref{item:light_tailed}.
Families with uniformly sub-exponential tails have finite moments of all orders, and the corresponding moment bounds are controlled uniformly \citep[Section 2.7]{vershynin2026high}.
Consequently, many regular exponential-family models satisfy the required finite-order moment condition under the usual compact-parameter assumptions.
Certain heavy-tailed families also satisfy the required finite moment condition, provided their tail parameters are sufficiently favorable.
Explicitly, for the log-normal distribution and the Pareto distribution, the moments of order $k$ are
\begin{align*}
    (\text{Log-normal}) &  &                 &
    z \sim e^{\mathcal{N}(\mu,\sigma^2)}
                        &  & \mathbb{E}[z^k]
    =
    \exp\qty(k\mu+\frac{k^2\sigma^2}{2}),      \\[0.5em]
    (\text{Pareto})     &  &                 &
    z \sim \mathrm{Pareto}(x_m,\alpha)
                        &  & \mathbb{E}[z^k]
    =
    \frac{\alpha}{\alpha-k}x_m^k,
    \quad k<\alpha,
\end{align*}
which are finite under appropriate parameter choices.
See also \citep{halliwell2015lognormal}.
We next check the estimation condition stated as \cref{item:identifiability}.
As an example, consider the Gaussian location model $z \sim \mathcal{N}(\mu,\sigma^2)$ where $\sigma$ is known and positive, and $\beta_0=\mu$.
The model is identifiable, and the sample mean estimator
\begin{equation*}
    \Pred(\{z_i\}_{i=1}^b) = \widehat{\mu}_b \coloneqq \frac{1}{b} \sum_{i=1}^b z_i
\end{equation*}
satisfies
\begin{equation*}
    \mathbb{E}[\norm{\widehat{\mu}_b-\mu}^2] = \frac{\sigma^2}{b}
\end{equation*}
by the definition of the variance, so \cref{item:identifiability} holds with $\sigma_{\max} = \sigma$.
More generally, the asymptotic efficiency and convergence of the maximum likelihood estimator (MLE) in regular parametric models motivate the estimator condition in \cref{item:identifiability}.
We emphasize that the particular convergence rate, $\order{b^{-1}}$, is not necessary.
Even when the convergence rate is much slower, the convergence guarantee still holds.

The main exclusions are cases such as exponentially growing objectives, distributions with unbounded parameter derivatives, and combinations of losses and distributions for which the required expectations do not exist.

\section{Proof of Proposition 1 (Deriving Technical Conditions)}
\label{app:proof_prop1}

In this section, we prove \cref{prop:technical_conditions}, which states that \cref{asm:main} implies the technical conditions \cref{asm:PL_bounded,asm:ell_theta,asm:beta_theta,asm:G_Lipschitz,asm:ell_bounded,asm:Fisher}.

We first record a consequence of the uniform moment bound \cref{eq:uniform_moment_bound} in Appendix~\ref{app:exact_formulation_of_assumptions}.

\begin{lemma}
    \label{lem:uniform-polynomial-moments}
    Suppose \cref{asm:main} holds.
    Let $P\colon \mathbb{R} \to \mathbb{R}$ be a polynomial of degree at most $q_{\max}$.
    Then
    \begin{equation*}
        \sup_{\beta_0 \in B} \mathbb{E}_{z \sim \mathcal{D}_{\beta_0}} \qty[\abs{P(\norm{z})}] < \infty.
    \end{equation*}
\end{lemma}
\begin{proof}
    Write $P(r) = \sum_{k = 0}^{q_{\max}} a_k r^k$.
    Then, by \cref{eq:uniform_moment_bound},
    \begin{equation*}
        \sup_{\beta_0 \in B} \mathbb{E}_{z \sim \mathcal{D}_{\beta_0}} \qty[\abs{P(\norm{z})}]
        \le
        \sum_{k = 0}^{q_{\max}} \abs{a_k} \sup_{\beta_0 \in B} \mathbb{E}_{z \sim \mathcal{D}_{\beta_0}} \qty[\norm{z}^k]
        \leq \sum_{k = 0}^{q_{\max}} \abs{a_k} M_k
        < \infty,
    \end{equation*}
    which concludes the proof.
\end{proof}

Using this lemma, we can verify the technical conditions.

\begin{proposition}[Restated]
    \propVerification
\end{proposition}
\begin{proof}
    We verify the required technical conditions one by one.
    First, we prove \cref{asm:PL_bounded}.
    By \cref{item:poly_ell} and \cref{lem:uniform-polynomial-moments},
    \begin{equation*}
        \sup_{\theta \in \Theta,\ \beta_0 \in B}
        \mathbb{E}_{z \sim \mathcal{D}_{\beta_0}} [\abs{\ell(z;\theta)}]
        < \infty.
    \end{equation*}
    Since $\beta(\Theta) \subseteq B$ by \cref{item:B_compact_convex}, this uniform bound gives
    \begin{equation*}
        \mathcal{L}_* = \inf_{\theta \in \Theta} \mathcal{L}(\theta)
        = \inf_{\theta \in \Theta} \mathbb{E}_{z \sim \mathcal{D}_{\beta(\theta)}} [\ell(z;\theta)]
        \ge \inf_{\theta \in \Theta,\ \beta_0 \in B} \mathbb{E}_{z \sim \mathcal{D}_{\beta_0}} [\ell(z;\theta)]
        \ge -\qty(\sup_{\theta \in \Theta,\ \beta_0 \in B} \mathbb{E}_{z \sim \mathcal{D}_{\beta_0}} [\abs{\ell(z;\theta)}])
        > -\infty.
    \end{equation*}
    Thus \cref{asm:PL_bounded} holds.

    Next, we prove \cref{asm:ell_theta}.
    Define
    \begin{align*}
        L_{\ell,\theta}^{\mathrm{Lip}}
         & \coloneqq
        \sup_{\theta\in\Theta,\ \beta_0\in B}
        \left(
        \mathbb{E}_{z\sim\mathcal{D}_{\beta_0}}
        \left[
            \norm{\nabla_\theta\ell(z;\theta)}^2
            \right]
        \right)^{1/2}, \\
        L_{\ell,\theta}^{\mathrm{sm}}
         & \coloneqq
        \sup_{\theta\in\Theta,\ \beta_0\in B}
        \mathbb{E}_{z\sim\mathcal{D}_{\beta_0}}
        \left[
            \norm{\nabla_\theta^2\ell(z;\theta)}
            \right].
    \end{align*}
    By \cref{item:poly_ell}, the two integrands are uniformly bounded over $\theta \in \Theta$ by polynomials in $\norm{z}$ of degrees at most $2 q_4$ and $q_5$, respectively.
    Since $\max\{2 q_4, q_5\} \leq q_{\max}$ by \cref{eq:q_max_definition}, both quantities are finite by \cref{lem:uniform-polynomial-moments}.
    Given the convexity of $\Theta$ (\cref{item:Theta_compact_convex}), the mean-value theorem gives the Lipschitz and smoothness bounds required in \cref{asm:ell_theta} \citep[Proposition 2.2.1]{cobzasLipschitzFunctions2019}.

    Next, we prove \cref{asm:beta_theta}.
    Since $\beta$ is $C^3$ and $\Theta$ is compact, the following quantities are finite:
    \begin{equation*}
        L_{\beta,\theta}^{\mathrm{Lip}}
        \coloneqq
        \sup_{\theta\in\Theta}
        \norm{
            {\pdv{\beta}{\theta}}(\theta)
        }, \qquad
        L_{\beta,\theta}^{\mathrm{sm}}
        \coloneqq
        \sup_{\theta\in\Theta}
        \norm{
            {\pdv[2]{\beta}{\theta}}(\theta)
        }, \qquad
        L_{\beta,\theta}^{\mathrm{H}}
        \coloneqq
        \sup_{\theta\in\Theta}
        \norm{
            {\pdv[3]{\beta}{\theta}}(\theta)
        }.
    \end{equation*}
    By the mean-value theorem and the convexity of $\Theta$, these bounds imply \cref{asm:beta_theta}.

    Next, we prove \cref{asm:G_Lipschitz}.
    Recall from \cref{eq:score_function,eq:G_definition} that
    \begin{equation*}
        s(z;\beta_0)
        =
        \nabla_{\beta} \ln p(z;\beta_0),
        \qquad
        G(\theta;\beta_0)
        =
        \mathbb{E}_{z\sim\mathcal{D}_{\beta_0}}
        \left[
            \ell(z;\theta)s(z;\beta_0)
            \right].
    \end{equation*}
    By \cref{item:poly_score,item:poly_ell}, the norm of $s(z;\beta_0)\nabla_\theta\ell(z;\theta)^\top$ is bounded uniformly over $\theta \in \Theta$ by a polynomial in $\norm{z}$ of degree at most $q_1+q_4$.
    Since $q_1+q_4 \leq q_{\max}$ by \cref{eq:q_max_definition}, this polynomial is integrable under every $\mathcal{D}_{\beta_0}$ by \cref{lem:uniform-polynomial-moments}.
    Thus, the dominated convergence theorem justifies differentiation under the integral sign with respect to $\theta$, and
    \begin{equation}
        \pdv{\theta} G(\theta;\beta_0)
        =
        \mathbb{E}_{z\sim\mathcal{D}_{\beta_0}}
        \left[
            s(z;\beta_0) \nabla_\theta\ell(z;\theta)^\top
            \right].
        \label{eq:G_theta_derivative}
    \end{equation}
    By \cref{item:regularity}, differentiation under the integral sign with respect to $\beta_0$ is valid, so
    \begin{align}
        \pdv{\beta} G(\theta;\beta_0) & = \pdv{\beta} \int \ell(z;\theta)s(z;\beta_0)p(z;\beta_0) \dd{\mu(z)}                                                                              &  & (\text{by \cref{eq:G_definition}})\notag                          \\
                                      & = \int \ell(z;\theta) \left\{ \pdv{\beta} s(z;\beta_0) + s(z;\beta_0)s(z;\beta_0)^\top \right\} p(z;\beta_0) \dd{\mu(z)}                           &  & (\nabla_{\beta} p(z;\beta_0)=p(z;\beta_0)s(z;\beta_0))     \notag \\
                                      & = \mathbb{E}_{z\sim\mathcal{D}_{\beta_0}} \left[ \ell(z;\theta) \left\{ \pdv{\beta} s(z;\beta_0) + s(z;\beta_0)s(z;\beta_0)^\top \right\} \right].
        \label{eq:G_beta_derivative}
    \end{align}
    By \cref{item:poly_score,item:poly_ell}, the norms of the integrands in \cref{eq:G_theta_derivative,eq:G_beta_derivative} are uniformly bounded over $(\theta,\beta_0) \in \Theta \times B$ by polynomials in $\norm{z}$ of degrees at most $q_1 + q_4$ and $\max\{q_2 + q_3, 2 q_1 + q_3\}$, respectively.
    Since both degrees are at most $q_{\max}$ by \cref{eq:q_max_definition}, \cref{lem:uniform-polynomial-moments} gives
    \begin{equation*}
        \sup_{\theta\in\Theta,\ \beta_0\in B} \norm{\pdv{\theta} G(\theta;\beta_0)} < \infty,
        \qquad
        \sup_{\theta\in\Theta,\ \beta_0\in B} \norm{\pdv{\beta} G(\theta;\beta_0)} < \infty.
    \end{equation*}
    Since $\Theta\times B$ is convex, the mean-value theorem implies that $G$ is Lipschitz on $\Theta\times B$.
    This proves \cref{asm:G_Lipschitz}.

    Next, we prove \cref{asm:ell_bounded}.
    By \cref{item:poly_ell}, $\abs{\ell(z;\theta)}$ and $\ell(z;\theta)^2$ are uniformly bounded over $\theta \in \Theta$ by polynomials in $\norm{z}$ of degrees at most $q_3$ and $2 q_3$, respectively.
    Since $2 q_3 \leq q_{\max}$ by \cref{eq:q_max_definition}, \cref{lem:uniform-polynomial-moments} gives
    \begin{equation*}
        M_\ell
        \coloneqq
        \sup_{\theta\in\Theta,\ \beta_0\in B}
        \mathrm{Var}_{z \sim \mathcal{D}_{\beta_0}} \left[ \ell(z; \theta) \right]
        =
        \sup_{\theta\in\Theta,\ \beta_0\in B}
        \qty(\mathbb{E}_{z\sim\mathcal{D}_{\beta_0}}
        \left[
            \ell(z;\theta)^2
            \right]
        - \mathbb{E}_{z\sim\mathcal{D}_{\beta_0}}
        \left[
            \ell(z;\theta)
            \right]^2)
        < \infty.
    \end{equation*}
    Thus \cref{asm:ell_bounded} holds.

    Finally, we prove \cref{asm:Fisher}.
    The trace of the Fisher information matrix defined in \cref{eq:Fisher_information} satisfies
    \begin{equation}
        \Tr(I(\beta_0))
        =
        \Tr(\mathbb{E}_{z\sim\mathcal{D}_{\beta_0}}
        \left[
            s(z;\beta_0)s(z;\beta_0)^\top
            \right])
        = \mathbb{E}_{z\sim\mathcal{D}_{\beta_0}}
        \left[\norm{s(z;\beta_0)}^2\right]. \label{eq:Fisher_trace}
    \end{equation}
    By \cref{item:poly_score}, $\norm{s(z;\beta_0)}^2$ has at most polynomial growth of degree at most $2 q_1$.
    Since $2 q_1 \leq q_{\max}$ by \cref{eq:q_max_definition}, \cref{lem:uniform-polynomial-moments,eq:Fisher_trace} gives
    \begin{equation*}
        I_{\max}
        \coloneqq
        \sup_{\beta_0\in B}
        \Tr(I(\beta_0))
        =
        \sup_{\beta_0\in B}
        \mathbb{E}_{z\sim\mathcal{D}_{\beta_0}}
        \left[
            \norm{s(z;\beta_0)}^2
            \right]
        < \infty.
    \end{equation*}
    Thus \cref{asm:Fisher} holds.
    This completes the verification.
\end{proof}

\section{Proof of Proposition 2 (Smoothness of the Performative Loss)}
\label{app:proof_prop2}

In this section, we prove \cref{prop:PL_LSmooth}, which states the sufficient conditions for the $\Lsm$-smoothness of the performative loss $\mathcal{L}$.
We start by decomposing the gradient difference into two terms:
\begin{align*}
    \norm{\nabla \mathcal{L}(\theta) - \nabla \mathcal{L}(\theta')}
    &= \norm{
        (\nabla_1\mathcal{L}(\theta) + \nabla_2\mathcal{L}(\theta)) - (\nabla_1\mathcal{L}(\theta') + \nabla_2\mathcal{L}(\theta'))
    } \\
    &\leq \norm{\nabla_1\mathcal{L}(\theta) - \nabla_1\mathcal{L}(\theta')} + \norm{\nabla_2\mathcal{L}(\theta) - \nabla_2\mathcal{L}(\theta')}.
\end{align*}
We will bound each term separately.
To this end, we start by deriving some useful lemmas for the subsequent analysis.
Differentiation under the integral sign with respect to $\beta_0$ is justified by \cref{item:regularity}, and the exchanges of integration order below are justified by Tonelli's theorem for nonnegative integrands.

\begin{lemma}\label{lem:grad_ell_lip_beta}
    Suppose that \cref{asm:main} holds.
    For any $\theta \in \Theta$ and $\beta_0 = \beta(\theta), \beta_1 \in B$, we have
    \begin{equation*}
        \norm{
            \mathbb{E}_{z \sim \mathcal{D}_{\beta_1}}\qty[\nabla_{\theta} \ell(z; \theta)]
            -
            \mathbb{E}_{z \sim \mathcal{D}_{\beta_0}}\qty[\nabla_{\theta} \ell(z; \theta)]
        }
        \le
        L_{\ell,\theta}^{\mathrm{Lip}} \sqrt{I_{\max}} \norm{\beta_1 - \beta_0}.
    \end{equation*}
\end{lemma}
\begin{proof}
    Since $B$ is convex (\cref{item:B_compact_convex}), $\beta_t \coloneqq \beta_0 + t(\beta_1 - \beta_0)$ is in $B$ for $t\in[0,1]$, and we have
    \begin{align}
               & \norm{
            \mathbb{E}_{z \sim \mathcal{D}_{\beta_1}}\qty[\nabla_{\theta} \ell(z; \theta)]
            -
            \mathbb{E}_{z \sim \mathcal{D}_{\beta_0}}\qty[\nabla_{\theta} \ell(z; \theta)]
        }                                                                                                                                                                                                                         \notag \\
        ={}    &
        \norm{
            \int
            \nabla_{\theta} \ell(z; \theta)
            \qty(p(z;\beta_1)-p(z;\beta_0))
            \dd{\mu(z)}
        }      &                                                                                                                                          &
        (\text{definitions})
        \notag                                                                                                                                                                                                                           \\
        ={}    & \norm{\int \nabla_{\theta} \ell(z; \theta) \int_0^1 \nabla_{\beta} p(z;\beta_t)^\top (\beta_1-\beta_0) \dd{t} \dd{\mu(z)}}               &   & (\text{fundamental theorem of calculus}) \notag                          \\
        \leq{} & \int \norm{\nabla_{\theta} \ell(z; \theta)} \int_0^1 \norm{\nabla_{\beta} p(z;\beta_t)} \norm{\beta_1-\beta_0} \dd{t} \dd{\mu(z)}        &   & (\text{triangle inequality})           \notag                            \\
        ={}    & \qty(\int \norm{\nabla_{\theta} \ell(z; \theta)} \int_0^1 \norm{\nabla_{\beta} p(z;\beta_t)} \dd{t} \dd{\mu(z)}) \norm{\beta_1-\beta_0}. &   & (\text{constancy of $\beta_1-\beta_0$}) \label{eq:grad_ell_lip_beta}
    \end{align}
    For the integral part, we have
    \begin{align}
               & \int \norm{\nabla_{\theta} \ell(z; \theta)} \int_0^1 \norm{\nabla_{\beta} p(z;\beta_t)} \dd{t} \dd{\mu(z)}                                                                                                                                                         \notag                                                                                   \\
        ={}    & \int_0^1 \int \norm{\nabla_{\theta} \ell(z; \theta)} \norm{\nabla_{\beta} p(z;\beta_t)} \dd{\mu(z)} \dd{t}                                                                                                                                                                &  & (\text{Tonelli's theorem})                         \notag                    \\
        ={}    & \int_0^1 \int \norm{\nabla_{\theta} \ell(z; \theta)} \norm{\nabla_{\beta} \ln p(z;\beta_t)} p(z;\beta_t) \dd{\mu(z)} \dd{t}                                                                                                                                               &  & (\text{log-derivative})                           \notag                     \\
        ={}    & \int_0^1 \mathbb{E}_{z \sim \mathcal{D}_{\beta_t}} \left[ \norm{\nabla_{\theta} \ell(z; \theta)} \norm{\nabla_{\beta} \ln p(z;\beta_t)} \right] \dd{t}                                                                                                                    &  & (\text{definition of expectation})           \notag                          \\
        \leq{} & \int_0^1 \mathbb{E}_{z \sim \mathcal{D}_{\beta_t}} \left[\norm{\nabla_{\theta} \ell(z; \theta)}^2\right]^{1/2} \mathbb{E}_{z \sim \mathcal{D}_{\beta_t}} \left[ \norm{\nabla_{\beta} \ln p(z;\beta_t)}^2\right]^{1/2} \dd{t}                                              &  & (\text{Cauchy--Schwarz inequality})            \notag                        \\
        \leq{} & \int_0^1 L_{\ell,\theta}^{\mathrm{Lip}} \sqrt{\Tr(I(\beta_t))} \dd{t}.                                                                                                                                                                                                    &  & (\text{\cref{asm:ell_theta,eq:Fisher_trace}}) \label{eq:integral_part_bound}
    \end{align}
    Thus, we have
    \begin{align*}
        \norm{
            \mathbb{E}_{z \sim \mathcal{D}_{\beta_1}}\qty[\nabla_{\theta} \ell(z; \theta)]
            -
            \mathbb{E}_{z \sim \mathcal{D}_{\beta_0}}\qty[\nabla_{\theta} \ell(z; \theta)]
        }
         & \leq \qty(\int_0^1 L_{\ell,\theta}^{\mathrm{Lip}} \sqrt{\Tr(I(\beta_t))} \dd{t}) \norm{\beta_1 - \beta_0} &  & (\text{by \cref{eq:grad_ell_lip_beta,eq:integral_part_bound}}) \\
         & \leq L_{\ell,\theta}^{\mathrm{Lip}} \sqrt{I_{\max}} \norm{\beta_1 - \beta_0},                             &  & (\text{\cref{asm:Fisher}})
    \end{align*}
    which concludes the proof.
\end{proof}

\begin{lemma}\label{lem:G_norm_bound}
    Suppose that \cref{asm:main} holds.
    Then, for all $\theta \in \Theta$ and $\beta_0 \in B$, we have
    \begin{equation*}
        \norm{G(\theta; \beta_0)} \leq \sqrt{M_\ell I_{\max}}.
    \end{equation*}
\end{lemma}
\begin{proof}
    We can evaluate as follows:
    \begin{align*}
                & \norm{G(\theta; \beta_0)} \notag                                                                                                                                                                                                       \\
        ={}     & \norm{\mathbb{E}_{z \sim \mathcal{D}_{\beta_0}} \qty[(\ell(z; \theta) - \bar{\ell}(\theta;\beta_0)) s(z;\beta_0)]}                                                          &  & (\text{by \cref{eq:G}})             \notag            \\
        \leq{}  & \mathbb{E}_{z \sim \mathcal{D}_{\beta_0}} \qty[\norm{\ell(z; \theta) - \bar{\ell}(\theta;\beta_0)} \norm{s(z;\beta_0)}]                                                     &  & (\text{triangle inequality})  \notag                  \\
        \leq {} & \sqrt{\mathbb{E}_{z \sim \mathcal{D}_{\beta_0}} \qty[(\ell(z;\theta) - \bar{\ell}(\theta;\beta_0))^2]\mathbb{E}_{z \sim \mathcal{D}_{\beta_0}} \qty[\norm{s(z;\beta_0)}^2]} &  & (\text{the Cauchy--Schwarz inequality})  \notag       \\
        = {}    & \sqrt{\mathrm{Var}_{z \sim \mathcal{D}_{\beta_0}} \qty[\ell(z; \theta)]\Tr(I(\beta_0))}                                                                                     &  & (\text{by \cref{eq:baseline,eq:Fisher_trace}}) \notag \\
        \leq{}  & \sqrt{M_\ell I_{\max}},                                                                                                                                                     &  & (\text{\cref{asm:ell_bounded,asm:Fisher}})
    \end{align*}
    which concludes the proof.
\end{proof}

Next, using these lemmas, we bound the gradient differences of $\nabla_1 \mathcal{L}$ and $\nabla_2 \mathcal{L}$, respectively.

\begin{lemma}\label{lem:PL_LSmooth_1}
    Suppose that \cref{asm:main} holds.
    Then, for all $\theta, \theta' \in \Theta$, we have
    \begin{equation*}
        \norm{\nabla_1 \mathcal{L}(\theta)-\nabla_1 \mathcal{L}(\theta')} \leq \qty(L_{\ell,\theta}^{\mathrm{sm}} + L_{\ell,\theta}^{\mathrm{Lip}} L_{\beta,\theta}^{\mathrm{Lip}} \sqrt{I_{\max}}) \norm{\theta-\theta'}.
    \end{equation*}
\end{lemma}
\begin{proof}
    By \cref{eq:nabla_L1} and the triangle inequality, we have
    \begin{align}
               & \norm{\nabla_1 \mathcal{L}(\theta)-\nabla_1 \mathcal{L}(\theta')} \notag                                                                                                                                                                                                                                                                                                                          \\
        ={}    & \norm{\mathbb{E}_{z \sim \mathcal{D}_{\beta(\theta)}}[\nabla_{\theta} \ell(z; \theta)] - \mathbb{E}_{z \sim \mathcal{D}_{\beta(\theta')}}[\nabla_{\theta} \ell(z; \theta')]} \notag                                                                                                                                                                                                               \\
        \leq{} & \norm{\mathbb{E}_{z \sim \mathcal{D}_{\beta(\theta)}}[\nabla_{\theta} \ell(z; \theta)] - \mathbb{E}_{z \sim \mathcal{D}_{\beta(\theta')}}[\nabla_{\theta} \ell(z; \theta)]} + \norm{\mathbb{E}_{z \sim \mathcal{D}_{\beta(\theta')}}[\nabla_{\theta} \ell(z; \theta)] - \mathbb{E}_{z \sim \mathcal{D}_{\beta(\theta')}}[\nabla_{\theta} \ell(z; \theta')]} \label{eq:PL_LSmooth_1_decomposition}
    \end{align}
    We first bound the first term of \cref{eq:PL_LSmooth_1_decomposition}.
    By \cref{lem:grad_ell_lip_beta}, we have
    \begin{align}
        \norm{
            \mathbb{E}_{z \sim \mathcal{D}_{\beta(\theta)}}\qty[\nabla_{\theta} \ell(z; \theta)]
            -
            \mathbb{E}_{z \sim \mathcal{D}_{\beta(\theta')}}\qty[\nabla_{\theta} \ell(z; \theta)]}
         & \leq L_{\ell, \theta}^{\mathrm{Lip}} \sqrt{I_{\max}} \norm{\beta(\theta) - \beta(\theta')}                  &  & (\text{\cref{lem:grad_ell_lip_beta}}) \notag \\
         & \leq L_{\ell, \theta}^{\mathrm{Lip}} L_{\beta, \theta}^{\mathrm{Lip}} \sqrt{I_{\max}} \norm{\theta-\theta'} &  & (\text{\cref{asm:beta_theta}})
        \label{eq:PL_LSmooth_1_1}
    \end{align}
    We next bound the second term of \cref{eq:PL_LSmooth_1_decomposition}.
    We have
    \begin{align}
               & \norm{\mathbb{E}_{z \sim \mathcal{D}_{\beta(\theta')}}[\nabla_{\theta} \ell(z; \theta)] - \mathbb{E}_{z \sim \mathcal{D}_{\beta(\theta')}}[\nabla_{\theta} \ell(z; \theta')]} \notag                                    \\
        \leq{} & \mathbb{E}_{z \sim \mathcal{D}_{\beta(\theta')}}\qty[\norm{\nabla_{\theta} \ell(z; \theta) - \nabla_{\theta} \ell(z; \theta')}]                                                      &  & (\text{Jensen's inequality})
        \notag                                                                                                                                                                                                                           \\
        \leq{} & L_{\ell,\theta}^{\mathrm{sm}} \norm{\theta-\theta'}.                                                                                                                                 &  & (\text{\cref{asm:ell_theta}})
        \label{eq:PL_LSmooth_1_2}
    \end{align}
    Thus, substituting \cref{eq:PL_LSmooth_1_1,eq:PL_LSmooth_1_2} into \cref{eq:PL_LSmooth_1_decomposition}, we obtain
    \begin{equation}
        \norm{\nabla_1 \mathcal{L}(\theta)-\nabla_1 \mathcal{L}(\theta')}
        \leq   \qty(L_{\ell,\theta}^{\mathrm{sm}} + L_{\ell,\theta}^{\mathrm{Lip}} L_{\beta,\theta}^{\mathrm{Lip}} \sqrt{I_{\max}}) \norm{\theta-\theta'},
        \label{eq:PL_LSmooth_1}
    \end{equation}
    which gives the desired bound of this lemma.
\end{proof}

\begin{lemma}\label{lem:PL_LSmooth_2}
    Suppose that \cref{asm:beta_theta,asm:G_Lipschitz,asm:ell_bounded,asm:Fisher} hold.
    Then, for all $\theta, \theta' \in \Theta$, we have
    \begin{equation*}
        \norm{\nabla_2 \mathcal{L}(\theta)-\nabla_2 \mathcal{L}(\theta')} \leq \qty(L_{\beta,\theta}^{\mathrm{sm}} \sqrt{M_\ell I_{\max}}+L_{\beta,\theta}^{\mathrm{Lip}}\qty(L_{G,\beta}^{\mathrm{Lip}} L_{\beta,\theta}^{\mathrm{Lip}}+L_{G,\theta}^{\mathrm{Lip}})) \norm{\theta-\theta'}.
    \end{equation*}
\end{lemma}
\begin{proof}
    By \cref{eq:nabla_L2}, we have
    \begin{align}
                 & \norm{\nabla_2 \mathcal{L}(\theta)-\nabla_2 \mathcal{L}(\theta')} \notag                                                              \\
        =     {} & \norm{
        {{\pdv{\beta}{\theta}}(\theta)}^\top G(\theta;\beta(\theta))
        -
        {{\pdv{\beta}{\theta}}(\theta')}^\top G(\theta';\beta(\theta'))
        }        &                                                                          & (\text{by \cref{eq:nabla_L2}})\notag                       \\
        \leq {}  &
        \norm{
        \qty({{\pdv{\beta}{\theta}}(\theta)}-{{\pdv{\beta}{\theta}}(\theta')})^\top G(\theta;\beta(\theta))
        }
        +
        \norm{
            {{\pdv{\beta}{\theta}}(\theta')}^\top
            \qty(
            G(\theta;\beta(\theta))-G(\theta';\beta(\theta'))
            )
        }        &                                                                          & (\text{triangle inequality})\notag                         \\
        \leq{}   &
        \norm{{{\pdv{\beta}{\theta}}(\theta)}-{{\pdv{\beta}{\theta}}(\theta')}} \, \norm{G(\theta;\beta(\theta))}
        +
        \norm{{{\pdv{\beta}{\theta}}(\theta')}} \,
        \norm{
            G(\theta;\beta(\theta))-G(\theta';\beta(\theta'))
        }.       &                                                                          & (\text{submultiplicativity}) \label{eq:beta_theta_G_bound}
    \end{align}
    We can bound the second norm in the second term of \cref{eq:beta_theta_G_bound} as
    \begin{align}
        \norm{
            G(\theta;\beta(\theta))-G(\theta';\beta(\theta'))
        }
                                                          & \le
        L_{G,\beta}^{\mathrm{Lip}} \norm{\beta(\theta)-\beta(\theta')}
        +
        L_{G,\theta}^{\mathrm{Lip}} \norm{\theta-\theta'} &     & (\text{\cref{asm:G_Lipschitz}}) \notag                       \\
                                                          & \le
        \qty(
        L_{G,\beta}^{\mathrm{Lip}} L_{\beta,\theta}^{\mathrm{Lip}}
        +
        L_{G,\theta}^{\mathrm{Lip}}
        )
        \norm{\theta-\theta'}.                            &     & (\text{\cref{asm:beta_theta}}) \label{eq:F_Lipschitz_conseq}
    \end{align}
    Then, by using \cref{asm:beta_theta}, \cref{lem:G_norm_bound}, \cref{asm:G_Lipschitz}, and \cref{eq:F_Lipschitz_conseq}, we can bound each term in \cref{eq:beta_theta_G_bound} as
    \begin{align*}
               & \norm{\nabla_2 \mathcal{L}(\theta)-\nabla_2 \mathcal{L}(\theta')}
        \notag                                                                                                                                                                                                                                                \\
        \leq{} & \qty(L_{\beta,\theta}^{\mathrm{sm}} \norm{\theta-\theta'}) \sqrt{M_\ell I_{\max}} + L_{\beta,\theta}^{\mathrm{Lip}} \qty(L_{G,\beta}^{\mathrm{Lip}} L_{\beta,\theta}^{\mathrm{Lip}}+L_{G,\theta}^{\mathrm{Lip}}) \norm{\theta-\theta'}\notag \\
        ={}    & \qty(L_{\beta,\theta}^{\mathrm{sm}} \sqrt{M_\ell I_{\max}}+L_{\beta,\theta}^{\mathrm{Lip}}\qty(L_{G,\beta}^{\mathrm{Lip}} L_{\beta,\theta}^{\mathrm{Lip}}+L_{G,\theta}^{\mathrm{Lip}})) \norm{\theta-\theta'},
    \end{align*}
    which gives the desired bound of this lemma.
\end{proof}

Finally, we prove \cref{prop:PL_LSmooth}, the main goal of this section.
\begin{proposition}[Restated]
    \propPLSmoothStatement[equation]
\end{proposition}
\begin{proof}
    By \cref{eq:nabla_L} and the triangle inequality, we have
    \begin{equation}
        \norm{\nabla \mathcal{L}(\theta)-\nabla \mathcal{L}(\theta')}
        \le \norm{\nabla_1 \mathcal{L}(\theta)-\nabla_1 \mathcal{L}(\theta')} + \norm{\nabla_2 \mathcal{L}(\theta)-\nabla_2 \mathcal{L}(\theta')} \label{eq:PL_LSmooth_0}
    \end{equation}
    Substituting the bounds in \cref{lem:PL_LSmooth_1,lem:PL_LSmooth_2} into \cref{eq:PL_LSmooth_0}, we can derive
    \begin{equation*}
        \norm{\nabla \mathcal{L}(\theta) - \nabla \mathcal{L}(\theta')}
        \leq \qty(L_{\ell,\theta}^{\mathrm{sm}} + L_{\ell,\theta}^{\mathrm{Lip}} L_{\beta,\theta}^{\mathrm{Lip}} \sqrt{I_{\max}} + L_{\beta,\theta}^{\mathrm{sm}} \sqrt{M_\ell I_{\max}} + L_{\beta,\theta}^{\mathrm{Lip}} \qty(L_{G,\beta}^{\mathrm{Lip}} L_{\beta,\theta}^{\mathrm{Lip}}+L_{G,\theta}^{\mathrm{Lip}}))\norm{\theta-\theta'},
    \end{equation*}
    which concludes the proof.
\end{proof}

\section{Proof of Proposition 3 (Gradient-Mapping Bound)}
\label{app:proof_prop3}

In this section, we prove \cref{prop:bound_gradient_norm}, which states the general convergence guarantee of the projected gradient method for minimizing the performative loss $\mathcal{L}$.
We use the notation from the main text.

We first state the so-called projection theorem \citep[Proposition 1.1.9]{bertsekas2009convex}.
Recall that $\Theta_\delta$ is a shrunk nonempty closed convex subset of $\Theta$, and $\Proj(v)$ is the Euclidean projection of $v$ onto $\Theta_\delta$.
\begin{lemma}\label{lem:projection-theorem}
    For any $v \in \mathbb{R}^n$ and $x \in \Theta_\delta$, we have
    \begin{equation*}
        \qty(v - \Proj(v))^\top \qty(x - \Proj(v)) \leq 0.
    \end{equation*}
\end{lemma}
\begin{proof}
    Note that $\Proj(v) \in \Theta_\delta$ is the solution to the following optimization problem:
    \begin{equation*}
        \min_{x \in \Theta_\delta} f(x) = \frac{1}{2} \norm{x-v}^2.
    \end{equation*}
    By the convexity of $\Theta_\delta$ and the optimality condition, for any $x \in \Theta_\delta$, we have
    \begin{equation*}
        \nabla f(\Proj(v))^\top (x - \Proj(v)) \geq 0.
    \end{equation*}
    Since $\nabla f(\Proj(v)) = \Proj(v) - v$, we obtain the desired inequality.
\end{proof}

We next derive a key bound regarding the gradient mapping, which is used in the convergence analysis.
Recall that the projected gradient method updates the parameter as
\begin{equation*}
    \theta_{t+1} = \Proj\qty(\theta_t - \alpha \widehat{\nabla \mathcal{L}}(\theta_t)),
\end{equation*}
and the true and estimated gradient mappings are defined as
\begin{align}
    \mathcal{G}_{\alpha}(\theta_t)
     & =
    \frac{1}{\alpha}
    \qty(\theta_t - \Proj\qty(\theta_t - \alpha \nabla \mathcal{L}(\theta_t))), \label{eq:true-gradient-mapping} \\
    \widehat{\mathcal{G}_{\alpha}}(\theta_t)
     & =
    \frac{1}{\alpha}
    \qty(\theta_t - \Proj\qty(\theta_t - \alpha \widehat{\nabla \mathcal{L}}(\theta_t)))
    = \frac{1}{\alpha} (\theta_t - \theta_{t+1}). \label{eq:estimated-gradient-mapping}
\end{align}

\begin{lemma}\label{lem:gradmap-main-bound}
    For any $t$, we have
    \begin{equation*}
        \norm{\mathcal{G}_\alpha(\theta_t)}^2 - 2 \nabla \mathcal{L}(\theta_t)^\top \widehat{\mathcal{G}_\alpha}(\theta_t) + \norm{\widehat{\mathcal{G}_\alpha}(\theta_t)}^2
        \leq \norm{\nabla \mathcal{L}(\theta_t)-\widehat{\nabla \mathcal{L}}(\theta_t)}^2.
    \end{equation*}
\end{lemma}
\begin{proof}
    Since $\theta_t \in \Theta_\delta$ and $\Proj(v) \in \Theta_\delta$ for every $v \in \mathbb{R}^n$, applying \cref{lem:projection-theorem} gives the following two inequalities:
    \begin{align*}
         & (\theta_t - \alpha \nabla \mathcal{L}(\theta_t) - \Proj(\theta_t - \alpha \nabla \mathcal{L}(\theta_t)))^\top (\theta_{t} - \Proj(\theta_t - \alpha \nabla \mathcal{L}(\theta_t))) \leq 0,                                                     \\
         & (\theta_t-\alpha \widehat{\nabla \mathcal{L}}(\theta_t)-\Proj(\theta_t-\alpha \widehat{\nabla \mathcal{L}}(\theta_t)))^\top (\Proj(\theta_t-\alpha \nabla \mathcal{L}(\theta_t))-\Proj(\theta_t-\alpha \widehat{\nabla \mathcal{L}}(\theta_t))
        ) \leq 0.
    \end{align*}
    Define $D_t \coloneqq \mathcal{G}_\alpha(\theta_t) - \widehat{\mathcal{G}_\alpha}(\theta_t)$.
    By \cref{eq:estimated-gradient-mapping,eq:true-gradient-mapping} and $\alpha>0$, we can rewrite the above inequalities as
    \begin{align}
        \qty(\mathcal{G}_{\alpha}(\theta_t) - \nabla \mathcal{L}(\theta_t))^\top \mathcal{G}_{\alpha}(\theta_t)
        = \mathcal{G}_{\alpha}(\theta_t)^\top \qty(\mathcal{G}_{\alpha}(\theta_t) - \nabla \mathcal{L}(\theta_t))
        &\leq 0, \label{eq:gradmap-main-bound-1}                                                                \\
        \qty(\widehat{\mathcal{G}_{\alpha}}(\theta_t) - \widehat{\nabla \mathcal{L}}(\theta_t))^\top (\widehat{\mathcal{G}_{\alpha}}(\theta_t)-\mathcal{G}_{\alpha}(\theta_t))
        = \qty(\widehat{\nabla \mathcal{L}}(\theta_t) - \widehat{\mathcal{G}_{\alpha}}(\theta_t))^\top D_t
        &\leq 0. \label{eq:gradmap-main-bound-2}
    \end{align}
    Then, we have
    \begin{align*}
               & \norm{\mathcal{G}_\alpha(\theta_t)}^2 - 2 \nabla \mathcal{L}(\theta_t)^\top \widehat{\mathcal{G}_\alpha}(\theta_t) + \norm{\widehat{\mathcal{G}_\alpha}(\theta_t)}^2                                                                                                                      \\
        ={}    & 2 \mathcal{G}_\alpha(\theta_t)^\top \qty(\mathcal{G}_\alpha(\theta_t) - \nabla \mathcal{L}(\theta_t)) + 2 \nabla \mathcal{L}(\theta_t)^\top D_t + \norm{\widehat{\mathcal{G}_\alpha}(\theta_t)}^2 - \norm{\mathcal{G}_\alpha(\theta_t)}^2                                                 \\
        \leq{} & 2 \nabla \mathcal{L}(\theta_t)^\top D_t + \norm{\widehat{\mathcal{G}_\alpha}(\theta_t)}^2 - \norm{\mathcal{G}_\alpha(\theta_t)}^2                                                                                                         &  & (\text{by \cref{eq:gradmap-main-bound-1}}) \\
        ={}    & 2\qty(\nabla \mathcal{L}(\theta_t) - \widehat{\mathcal{G}_\alpha}(\theta_t))^\top D_t - \norm{D_t}^2                                                                                                                                                                                      \\
        ={}    & 2\qty(\nabla \mathcal{L}(\theta_t) - \widehat{\nabla \mathcal{L}}(\theta_t))^\top D_t - \norm{D_t}^2 + 2\qty(\widehat{\nabla \mathcal{L}}(\theta_t) - \widehat{\mathcal{G}_\alpha}(\theta_t))^\top D_t                                                                                    \\
        \leq{} & 2\qty(\nabla \mathcal{L}(\theta_t) - \widehat{\nabla \mathcal{L}}(\theta_t))^\top D_t - \norm{D_t}^2                                                                                                                                      &  & (\text{by \cref{eq:gradmap-main-bound-2}}) \\
        \leq{} & \norm{\nabla \mathcal{L}(\theta_t)-\widehat{\nabla \mathcal{L}}(\theta_t)}^2.                                                                                                                                                             &  & (2a^\top b - \norm{b}^2 \leq \norm{a}^2)
    \end{align*}
    which concludes the proof.
\end{proof}

The bound in \cref{lem:gradmap-main-bound} is tight when $\Theta(=\Theta_\delta)=\mathbb{R}^n$, i.e., $\mathcal{G}_\alpha(\theta_t) = \nabla \mathcal{L}(\theta_t)$ and $\widehat{\mathcal{G}_\alpha}(\theta_t) = \widehat{\nabla \mathcal{L}}(\theta_t)$.

Now, we are ready to prove \cref{prop:bound_gradient_norm}.
For background on projected gradient mappings and related descent arguments, see \citep[Lemma 10.4, Lemma 10.14]{beckChapter10Proximal2017}. 

\begin{proposition}[Restated]
    \propBoundGradientNormStatement[equation]
\end{proposition}
\begin{proof}
    Under the $\Lsm$-smoothness of $\mathcal{L}$ in \cref{prop:PL_LSmooth}, the following inequality holds:
    \begin{align*}
        \mathcal{L}(\theta_{t+1}) & \le \mathcal{L}(\theta_t) + \nabla\mathcal{L}(\theta_t)^\top (\theta_{t+1}-\theta_t) + \frac{\Lsm}{2}\norm{\theta_{t+1}-\theta_t}^2 \notag                                                                                                                           \\
                                  & = \mathcal{L}(\theta_t) - \alpha \nabla \mathcal{L}(\theta_t)^\top \widehat{\mathcal{G}_{\alpha}}(\theta_t) + \frac{\Lsm \alpha^2}{2} \norm{\widehat{\mathcal{G}_{\alpha}}(\theta_t)}^2                    \notag &  & (\text{by \cref{eq:gradient_mapping_update}}) \\
                                  & \le \mathcal{L}(\theta_t) - \alpha\nabla\mathcal{L}(\theta_t)^\top \widehat{\mathcal{G}_\alpha}(\theta_t) + \frac{\alpha}{2}\norm{\widehat{\mathcal{G}_\alpha}(\theta_t)}^2  \notag                               &  & (\alpha \leq 1/\Lsm)                          \\
                                  & \le\mathcal{L}(\theta_t) - \frac{\alpha}{2}\norm{\mathcal{G}_\alpha(\theta_t)}^2 + \frac{\alpha}{2}\norm{\widehat{\nabla\mathcal{L}}(\theta_t) - \nabla \mathcal{L}(\theta_t)}^2,                                 &  & (\text{\cref{lem:gradmap-main-bound}})
    \end{align*}
    This yields
    \begin{equation*}
        \mathcal{L}(\theta_{t+1}) - \mathcal{L}(\theta_t) \le - \frac{\alpha}{2}\norm{\mathcal{G}_{\alpha}(\theta_t)}^2 + \frac{\alpha}{2}\norm{\nabla\mathcal{L}(\theta_t) - \widehat{\nabla\mathcal{L}}(\theta_t)}^2.
    \end{equation*}
    Summing over $t=0,1,\ldots,T-1$, we obtain
    \begin{equation*}
        \mathcal{L}(\theta_T)-\mathcal{L}(\theta_0) \le -\frac{\alpha}{2}\sum_{t=0}^{T-1}\norm{\mathcal{G}_{\alpha}(\theta_t)}^2+\frac{\alpha}{2}\sum_{t=0}^{T-1}\norm{\nabla\mathcal{L}(\theta_t) - \widehat{\nabla\mathcal{L}}(\theta_t)}^2.
    \end{equation*}
    Rearranging and using \cref{asm:PL_bounded} established in \cref{prop:technical_conditions}, we have
    \begin{equation*}
        \sum_{t=0}^{T-1}\norm{\mathcal{G}_{\alpha}(\theta_t)}^2
        \le \frac{2 \qty(\mathcal{L}(\theta_0)-\mathcal{L}(\theta_T))}{\alpha}+\sum_{t=0}^{T-1}\norm{\nabla\mathcal{L}(\theta_t) - \widehat{\nabla\mathcal{L}}(\theta_t)}^2
        \le \frac{2 \qty(\mathcal{L}(\theta_0)-\mathcal{L}_*)}{\alpha}+\sum_{t=0}^{T-1}\norm{\nabla\mathcal{L}(\theta_t) - \widehat{\nabla\mathcal{L}}(\theta_t)}^2.
    \end{equation*}
    Finally, taking expectations on both sides and dividing by $T$, the linearity of expectation yields
    \begin{equation*}
        \frac{1}{T} \sum_{t=0}^{T-1} \mathbb{E} \left[\norm{\mathcal{G}_{\alpha}(\theta_t)}^2 \right]
        \le \frac{2 \qty(\mathcal{L}(\theta_0)-\mathcal{L}_*)}{\alpha T}+\frac{1}{T}\sum_{t=0}^{T-1}\mathbb{E}\left[\norm{\nabla\mathcal{L}(\theta_t) - \widehat{\nabla\mathcal{L}}(\theta_t)}^2\right].
    \end{equation*}
    This concludes the proof.
\end{proof}

\section{Proof of Proposition 4 (Gradient-Estimation Error Bound)}
\label{app:proof_prop4}

In this section, we prove \cref{prop:bound_L1_L2}, which gives a bound on the estimation error of the gradient estimator used in \cref{alg:performative_prediction_1st_order}.
We first establish the following lemma.
\begin{lemma}\label{lem:estimated_Jacobian_error}
    Suppose that \cref{asm:main} holds, $b \geq b_{\min}$, and $\theta_t \in \Theta_\delta$.
    Then the finite-difference Jacobian estimator satisfies the following bound:
    \begin{equation*}
        \mathbb{E}\left[\norm{{\pdv{\beta}{\theta}}(\theta_t) - \widehat{\pdv{\beta}{\theta}}(\theta_t)}^2\right]
        \leq
        n \qty(
        \frac{(L_{\beta,\theta}^{\mathrm{H}})^2\delta^4}{12}
        +
        \frac{3\sigma_{\max}^2}{2\delta^2 b}
        ).
    \end{equation*}
\end{lemma}
\begin{proof}
    Fix an arbitrary $i$.
    For $\delta > 0$ and the $i$-th standard basis vector $e_i \in \mathbb{R}^n$, define $\theta_{t,i}^{+} \coloneqq \theta_t + \delta e_i$ and $\theta_{t,i}^{-} \coloneqq \theta_t - \delta e_i$.
    Since the Hessian of $\beta$ is Lipschitz continuous by \cref{prop:technical_conditions,asm:beta_theta}, we can apply the following Taylor expansions:
    \begin{equation}
        \beta(\theta_{t,i}^{+})
        =
        \beta(\theta_t)
        + \delta {\pdv{\beta}{\theta}}(\theta_t)e_i
        + \frac{\delta^2}{2}{\pdv[2]{\beta}{\theta}}(\theta_t)[e_i,e_i]
        + r_i^+,
        \qquad
        \norm{r_i^+}
        \le
        \frac{L_{\beta,\theta}^{\mathrm{H}}}{6}\delta^3,
        \label{eq:taylor_plus}
    \end{equation}
    and
    \begin{equation}
        \beta(\theta_{t,i}^{-})
        =
        \beta(\theta_t)
        - \delta {\pdv{\beta}{\theta}}(\theta_t)e_i
        + \frac{\delta^2}{2}{\pdv[2]{\beta}{\theta}}(\theta_t)[e_i,e_i]
        + r_i^-,
        \qquad
        \norm{r_i^-}
        \le
        \frac{L_{\beta,\theta}^{\mathrm{H}}}{6}\delta^3.
        \label{eq:taylor_minus}
    \end{equation}
    Subtracting \cref{eq:taylor_minus} from \cref{eq:taylor_plus} and applying the remainder bounds gives
    \begin{equation}
        \norm{{\pdv{\beta}{\theta}}(\theta_t) e_i
            - \frac{\beta(\theta_{t,i}^{+}) - \beta(\theta_{t,i}^{-})}{2\delta}}
        =     \norm{\frac{r_i^- - r_i^+}{2\delta}}
        \leq  \frac{1}{2\delta}\qty(\norm{r_i^+}+\norm{r_i^-})
        \leq  \frac{L_{\beta,\theta}^{\mathrm{H}}}{6}\delta^2.
        \label{eq:finite_difference_error}
    \end{equation}
    Using \cref{eq:finite_difference_error}, we can bound the error between the $i$-th column of the true Jacobian and the $i$-th column of the Jacobian estimator as follows:
    \begin{align*}
               & \norm{{\pdv{\beta}{\theta}}(\theta_t) e_i - \widehat{\pdv{\beta}{\theta}}(\theta_t) e_i}^2                                                                                           \\
        ={}    & \norm{
            {\pdv{\beta}{\theta}}(\theta_t) e_i
            - \frac{\beta(\theta_{t,i}^{+}) - \beta(\theta_{t,i}^{-})}{2\delta}
            + \frac{\beta(\theta_{t,i}^{+}) - \beta(\theta_{t,i}^{-})}{2\delta}
            - \frac{\hat\beta_{t,i}^{+} - \hat\beta_{t,i}^{-}}{2\delta}
        }^2
               &                                                                                            & (\scalebox{0.93}{\text{by \refLine{line:1st_estimate} of \cref{alg:performative_prediction_1st_order}}}) \\
        ={}    & \norm{
            {\pdv{\beta}{\theta}}(\theta_t) e_i
            - \frac{\beta(\theta_{t,i}^{+}) - \beta(\theta_{t,i}^{-})}{2\delta}
            + \frac{1}{2\delta}\qty(\beta(\theta_{t,i}^{+})-\hat\beta_{t,i}^{+})
            + \frac{1}{2\delta}\qty(-\beta(\theta_{t,i}^{-})+\hat\beta_{t,i}^{-})
        }^2                                                                                                                                                                                           \\
        \leq{} & 3\norm{
            {\pdv{\beta}{\theta}}(\theta_t) e_i
            - \frac{\beta(\theta_{t,i}^{+}) - \beta(\theta_{t,i}^{-})}{2\delta}
        }^2
        + \frac{3}{4\delta^2}\norm{\beta(\theta_{t,i}^{+})-\hat\beta_{t,i}^{+}}^2
        + \frac{3}{4\delta^2}\norm{\beta(\theta_{t,i}^{-})-\hat\beta_{t,i}^{-}}^2
               &                                                                                            & (\scalebox{0.63}{$\norm{a+b+c}^2 \leq 3(\norm{a}^2+\norm{b}^2+\norm{c}^2)$})            \\
        \leq{} & \frac{(L_{\beta,\theta}^{\mathrm{H}})^2\delta^4}{12}
        + \frac{3}{4\delta^2}\norm{\beta(\theta_t+\delta e_i)-\hat\beta_{t,i}^{+}}^2
        + \frac{3}{4\delta^2}\norm{\beta(\theta_t-\delta e_i)-\hat\beta_{t,i}^{-}}^2.
               &                                                                                            & (\text{by \cref{eq:finite_difference_error}})
    \end{align*}
    Finally, we obtain:
    \begin{align*}
               & \mathbb{E}\left[\norm{{\pdv{\beta}{\theta}}(\theta_t) - \widehat{\pdv{\beta}{\theta}}(\theta_t)}^2\right] \leq \mathbb{E}\left[\norm{{\pdv{\beta}{\theta}}(\theta_t) - \widehat{\pdv{\beta}{\theta}}(\theta_t)}_{\mathrm{F}}^2\right]         &                                      & (\norm{\cdot}_{\mathrm{F}} \scalebox{0.85}{\text{ is Frobenius norm}}) \\
        ={}    & \sum_{i=1}^n \mathbb{E}\left[\norm{{\pdv{\beta}{\theta}}(\theta_t) e_i - \widehat{\pdv{\beta}{\theta}}(\theta_t) e_i}^2\right] &                                      & (\text{by definition})                                \\
        \leq{} & \sum_{i=1}^n \qty(
        \frac{(L_{\beta,\theta}^{\mathrm{H}})^2\delta^4}{12}
        + \frac{3}{4\delta^2} \mathbb{E}\left[\norm{\beta(\theta_t+\delta e_i) - \hat\beta_{t,i}^{+}}^2\right]
        + \frac{3}{4\delta^2} \mathbb{E}\left[\norm{\beta(\theta_t-\delta e_i) - \hat\beta_{t,i}^{-}}^2\right]
        )      &                                                                                                                                & (\text{by above inequality})                                                                 \\
        \leq{} & n \qty(
        \frac{(L_{\beta,\theta}^{\mathrm{H}})^2\delta^4}{12}
        +
        \frac{3\sigma_{\max}^2}{2\delta^2 b}
        ).
               &                                                                                                                                & (\text{\cref{item:identifiability}})
    \end{align*}
    This concludes the proof.
\end{proof}

Using \cref{lem:estimated_Jacobian_error}, we prove \cref{prop:bound_L1_L2}.

\begin{proposition}[Restated]
    \propBoundLOneLTwo[equation]
\end{proposition}
\begin{proof}
    We can decompose the norm to be bounded as follows:
    \begin{align}
               & \norm{\nabla \mathcal{L}(\theta_t) - \widehat{\nabla\mathcal{L}}(\theta_t)}^2 \notag                                                                                                                                                                \\
        ={}    & \norm{\qty(\widehat{\nabla_1\mathcal{L}} (\theta_t) - \nabla_1\mathcal{L} (\theta_t)) + \qty(\widehat{\nabla_2\mathcal{L}} (\theta_t) - \nabla_2\mathcal{L} (\theta_t))}^2 \notag &  & (\text{by \cref{eq:nabla_L}})                                \\
        \leq{} & \qty( \norm{\nabla_1\mathcal{L}(\theta_t) - \widehat{\nabla_1\mathcal{L}}(\theta_t)} + \norm{\nabla_2\mathcal{L}(\theta_t) - \widehat{\nabla_2\mathcal{L}}(\theta_t)} )^2.        &  & (\text{triangle inequality})\label{eq:bound_L1_L2_two_terms}
    \end{align}
    We first bound the term $\norm{\nabla_1\mathcal{L}(\theta_t) - \widehat{\nabla_1\mathcal{L}}(\theta_t)}$ in \cref{eq:bound_L1_L2_two_terms}.
    The estimate $\hat\beta_t$ computed at \refLine{line:1st_estimate} of \cref{alg:performative_prediction_1st_order} belongs to $B$, so \cref{lem:grad_ell_lip_beta} gives
    \begin{align}
        \norm{\nabla_1\mathcal{L}(\theta_t) - \widehat{\nabla_1\mathcal{L}}(\theta_t)}
        ={}    & \norm{\mathbb{E}_{z \sim \mathcal{D}_{\beta(\theta_t)}}\qty[\nabla_{\theta} \ell(z; \theta_t)] - \mathbb{E}_{z \sim \mathcal{D}_{\hat\beta_t}}\qty[\nabla_{\theta} \ell(z; \theta_t)]} &  & (\text{by \cref{eq:nabla_L1}}) \notag                                   \\
        \leq{} & L_{\ell,\theta}^{\mathrm{Lip}} \sqrt{I_{\max}} \norm{\beta(\theta_t) - \hat\beta_t}.                                                                                                   &  & (\text{\cref{lem:grad_ell_lip_beta}}) \label{eq:bound_L1_L2_first_term}
    \end{align}
    We next bound the term $\norm{\nabla_2\mathcal{L}(\theta_t) - \widehat{\nabla_2\mathcal{L}}(\theta_t)}$ in \cref{eq:bound_L1_L2_two_terms}.
    \begin{align}
               & \norm{\nabla_2\mathcal{L}(\theta_t) - \widehat{\nabla_2\mathcal{L}}(\theta_t)} \notag                                                                                                                                                                                                                                                                                                                               \\
        ={}    & \norm{{\pdv{\beta}{\theta}}(\theta_t)^\top G(\theta_t;\beta(\theta_t)) - \widehat{\pdv{\beta}{\theta}}(\theta_t)^\top G(\theta_t;\hat\beta_t)}                                                                                &                                                                 & (\scalebox{0.75}{\text{by \cref{eq:nabla_L2} and \refLine{line:1st_grad} of \cref{alg:performative_prediction_1st_order}}}) \notag \\
        \leq{} & \norm{{\pdv{\beta}{\theta}}(\theta_t)^\top \qty(G(\theta_t;\beta(\theta_t)) - G(\theta_t;\hat\beta_t))} + \norm{\qty({\pdv{\beta}{\theta}}(\theta_t) - \widehat{\pdv{\beta}{\theta}}(\theta_t))^\top G(\theta_t;\hat\beta_t)}
               &                                                                                                                                                                                                                               & (\text{triangle inequality})
        \notag                                                                                                                                                                                                                                                                                                                                                                                                                       \\
        \leq{} & \norm{{\pdv{\beta}{\theta}}(\theta_t)} \norm{G(\theta_t;\beta(\theta_t))-G(\theta_t;\hat\beta_t)} + \norm{{\pdv{\beta}{\theta}}(\theta_t)-\widehat{\pdv{\beta}{\theta}}(\theta_t)} \norm{G(\theta_t;\hat\beta_t)}.
               &                                                                                                                                                                                                                               & (\text{submultiplicativity}) \label{eq:bound_L1_L2_second_term}
    \end{align}
    For the first norm in the first term of \cref{eq:bound_L1_L2_second_term}, \cref{asm:beta_theta} and differentiability of $\beta$ imply the following bound \citep[Proposition 2.2.1]{cobzasLipschitzFunctions2019}:
    \begin{equation}
        \sup_{\theta \in \Theta} \norm{{\pdv{\beta}{\theta}}(\theta)} \le L_{\beta,\theta}^{\mathrm{Lip}}. \label{eq:Jacobian_norm_bounded_by_Lipschitz}
    \end{equation}
    Applying \cref{eq:Jacobian_norm_bounded_by_Lipschitz,asm:G_Lipschitz,lem:G_norm_bound} to the other three norms in \cref{eq:bound_L1_L2_second_term} gives
    \begin{equation}
        \norm{\nabla_2\mathcal{L}(\theta_t) - \widehat{\nabla_2\mathcal{L}}(\theta_t)}
        \leq   L_{\beta,\theta}^{\mathrm{Lip}} L_{G,\beta}^{\mathrm{Lip}} \norm{\beta(\theta_t)-\hat\beta_t} + \sqrt{M_\ell I_{\max}} \norm{{\pdv{\beta}{\theta}}(\theta_t)-\widehat{\pdv{\beta}{\theta}}(\theta_t)}. \label{eq:bound_L1_L2_second_term_final}
    \end{equation}

    Substituting the bounds in \cref{eq:bound_L1_L2_first_term,eq:bound_L1_L2_second_term_final} into \cref{eq:bound_L1_L2_two_terms}, and using $\norm{a+b+c}^2 \leq 3(\norm{a}^2 + \norm{b}^2 + \norm{c}^2)$, we can derive
    \begin{align}
               & \norm{\nabla \mathcal{L}(\theta_t) - \widehat{\nabla\mathcal{L}}(\theta_t)}^2\notag                                                                                                                                                                                                                        \\
        \leq{} & \qty(L_{\ell,\theta}^{\mathrm{Lip}} \sqrt{I_{\max}} \norm{\beta(\theta_t) - \hat\beta_t} + L_{\beta,\theta}^{\mathrm{Lip}} L_{G,\beta}^{\mathrm{Lip}} \norm{\beta(\theta_t)-\hat\beta_t} + \sqrt{M_\ell I_{\max}} \norm{{\pdv{\beta}{\theta}}(\theta_t)-\widehat{\pdv{\beta}{\theta}}(\theta_t)})^2 \notag \\
        \leq{} & 3\qty(\qty(L_{\ell,\theta}^{\mathrm{Lip}})^2 I_{\max} + \qty(L_{\beta,\theta}^{\mathrm{Lip}} L_{G,\beta}^{\mathrm{Lip}})^2) \norm{\beta(\theta_t)-\hat\beta_t}^2 + 3 M_\ell I_{\max} \norm{{\pdv{\beta}{\theta}}(\theta_t)-\widehat{\pdv{\beta}{\theta}}(\theta_t)}^2,
        \label{eq:bound_L1_L2_final}
    \end{align}
    By taking the expectation of both sides of \cref{eq:bound_L1_L2_final}, we obtain
    \begin{align*}
                          & \mathbb{E}\left[\norm{\nabla \mathcal{L}(\theta_t) - \widehat{\nabla\mathcal{L}}(\theta_t)}^2\right]                                                                                                                                                                                                                                                                                                                                                                                       \\
        \leq{}            & \mathrlap{3\qty(\qty(L_{\ell,\theta}^{\mathrm{Lip}})^2 I_{\max} + \qty(L_{\beta,\theta}^{\mathrm{Lip}} L_{G,\beta}^{\mathrm{Lip}})^2) \mathbb{E}\left[\norm{\beta(\theta_t)-\hat\beta_t}^2\right] + 3 M_\ell I_{\max} \mathbb{E}\left[\norm{{\pdv{\beta}{\theta}}(\theta_t)-\widehat{\pdv{\beta}{\theta}}(\theta_t)}^2\right] \qquad (\text{by \cref{eq:bound_L1_L2_final}})}                                                                                                              \\
        \leq{}            & 3\qty(\qty(L_{\ell,\theta}^{\mathrm{Lip}})^2 I_{\max} + \qty(L_{\beta,\theta}^{\mathrm{Lip}} L_{G,\beta}^{\mathrm{Lip}})^2) \frac{\sigma_{\max}^2}{b}
        + 3 M_\ell I_{\max} \cdot n \qty(
        \frac{(L_{\beta,\theta}^{\mathrm{H}})^2\delta^4}{12}
        +
        \frac{3\sigma_{\max}^2}{2\delta^2 b}
        )                 &                                                                                                                                                                                                                                                                                                                                                                               & (\scalebox{0.82}{\text{\cref{asm:main} and \cref{lem:estimated_Jacobian_error}}})                                           \\
        ={}               & 3\qty(\qty(L_{\ell,\theta}^{\mathrm{Lip}})^2 I_{\max} + \qty(L_{\beta,\theta}^{\mathrm{Lip}} L_{G,\beta}^{\mathrm{Lip}})^2) \frac{\sigma_{\max}^2}{b}
        + n \cdot
        3 M_\ell I_{\max}
        \qty(
        \frac{(L_{\beta,\theta}^{\mathrm{H}})^2}{12}
        +
        \frac{3\sigma_{\max}^2}{2}
        )
        \frac{1}{b^{2/3}} &                                                                                                                                                                                                                                                                                                                                                                               & (\delta = b^{-1/6})                                                                                        \\
        ={}               & \frac{C_1}{b} + \frac{n C_2}{b^{2/3}},                                                                                                                                                                                                                                                                                                                                        &                                                                  & (\scalebox{0.95}{\text{definitions of $C_1$ and $C_2$}})
    \end{align*}
    which concludes the proof.
\end{proof}

\section{
    \texorpdfstring{\scalebox{0.93}{Sphere-Smoothed Jacobian Estimation and Proof of Proposition 5}}{Sphere-Smoothed Jacobian Estimation and Proof of Proposition 5}
}
\label{app:sphere_smoothing}

In this section, we prove \cref{prop:sphere_smoothing}.

We first describe the sphere-smoothed variant of the Jacobian estimator used in \cref{alg:performative_prediction_1st_order}.
The coordinate-wise Jacobian estimator in \cref{alg:performative_prediction_1st_order} has essentially the same structure as a coordinate-wise two-point zeroth-order gradient estimator \citep{hikimaZerothorderGradientEstimators2025}.
More precisely, for the $i$-th column of the Jacobian of $\beta$, the estimator
\begin{equation*}
    \frac{\hat\beta_{t,i}^{+}-\hat\beta_{t,i}^{-}}{2\delta}
\end{equation*}
is a central finite-difference estimator of the corresponding partial derivative.
Thus, \cref{lem:estimated_Jacobian_error} coincides with the finite-difference error bound in \citet[Lemma 3]{hikimaZerothorderGradientEstimators2025}, up to notation and constants.
This observation suggests that the coordinate-wise finite-difference estimator can be replaced by a random-direction estimator based on smoothing on a sphere, which exhibits a more favorable dependence on the dimension $n$ \citep{hikimaZerothorderGradientEstimators2025}.

For this variant, define the sphere-shrunk feasible set by
\begin{equation}
    \Theta_\delta^{\mathrm{sp}}
    \coloneqq
    \left\{
    \theta \in \Theta
    \mathrel{}\middle|\mathrel{}
    \theta \pm \delta u \in \Theta
    \text{ for every } u \in \mathbb{S}^{n-1}
    \right\}.
    \label{eq:sphere_shrunk_feasible_set}
\end{equation}

At iteration $t$, draw independent random directions
\begin{equation*}
    u_t^1,\ldots,u_t^{b_{\mathrm{sp}}} \sim \mathrm{Unif}(\mathbb{S}^{n-1}),
\end{equation*}
where $b_{\mathrm{sp}} \in \mathbb{N}$ is the number of random directions and $\mathbb{S}^{n-1}$ denotes the unit sphere in $\mathbb{R}^n$.
For each direction $u_t^i$ ($1 \leq i \leq b_{\mathrm{sp}}$), deploy the two perturbed models and collect i.i.d.\ samples from the corresponding induced distributions ($j = 1,\ldots,b_{\min}$):
\begin{equation*}
    z_{t,i,j}^{+} \sim \mathcal{D}_{\beta(\theta_t+\delta u_t^i)}, \qquad z_{t,i,j}^{-} \sim \mathcal{D}_{\beta(\theta_t-\delta u_t^i)}.
\end{equation*}
Conditional on the sampled directions, the datasets collected for different directions and signs are mutually independent.
Here, $b_{\min} \in \mathbb{N}$ is the minimum number of samples introduced by \cref{item:identifiability}.
Then, we estimate the distribution parameters from these $b_{\min}$ samples:
\begin{equation*}
    \hat\beta_{t,i}^{+} = \Pred\left(\{z_{t,i,j}^{+}\}_{j=1}^{b_{\min}}\right),
    \qquad
    \hat\beta_{t,i}^{-} = \Pred\left(\{z_{t,i,j}^{-}\}_{j=1}^{b_{\min}}\right).
\end{equation*}
The sphere-smoothed Jacobian estimator is then defined by
\begin{equation}
    \widehat{\pdv{\beta}{\theta}}(\theta_t)
    =
    \frac{n}{b_{\mathrm{sp}}}
    \sum_{i=1}^{b_{\mathrm{sp}}}
    \frac{\hat\beta_{t,i}^{+}-\hat\beta_{t,i}^{-}}{2\delta}
    (u_t^i)^\top.
    \label{eq:sphere_jacobian_estimator}
\end{equation}
All other parts of the algorithm are unchanged, except that the projection is onto $\Theta_\delta^{\mathrm{sp}}$.

We impose the following additional assumptions on the sphere-smoothed Jacobian estimator.
\begin{assumption}\label{asm:sphere_smoothing}
    The sphere-shrunk feasible set $\Theta_\delta^{\mathrm{sp}}$ in \cref{eq:sphere_shrunk_feasible_set} is nonempty.
    The estimator $\hat\beta$ of the distribution parameter is conditionally unbiased given the model parameter $\theta$:
    \begin{equation*}
        \mathbb{E}[\hat\beta - \beta(\theta) \mid \theta] = 0.
    \end{equation*}
    For every $\theta, \theta' \in \Theta$, $b \geq b_{\min}$, and $1 \leq j \leq m$, the following component-wise strengthenings of \cref{item:identifiability} and \cref{asm:beta_theta} hold:
    \begin{equation*}
        \mathbb{E}\left[\abs{\hat\beta_j - \beta(\theta)_j}^2 \mid \theta\right] \le \frac{\sigma_{\max,j}^2}{b},
        \qquad
        \norm{\nabla_\theta^2 \beta_j(\theta)-\nabla_\theta^2 \beta_j(\theta')} \le L_{\beta,\theta,j}^{\mathrm{H}} \norm{\theta-\theta'}.
    \end{equation*}
    Here, $(\cdot)_j$ denotes the $j$-th component of a vector, and the constants $\sigma_{\max,j}^2$ and $L_{\beta,\theta,j}^{\mathrm{H}}$ satisfy
    \begin{equation}
        \sigma_{\max}^2 = \sum_{j=1}^m \sigma_{\max,j}^2,
        \qquad
        L_{\beta,\theta}^{\mathrm{H}} = \sqrt{\sum_{j=1}^m (L_{\beta,\theta,j}^{\mathrm{H}})^2}.
        \label{eq:componentwise_constants}
    \end{equation}
\end{assumption}

By \cref{item:Theta_compact_convex}, $\Theta_\delta^{\mathrm{sp}}$ is closed and convex.
Throughout this section, $\Proj$ denotes the Euclidean projection onto $\Theta_\delta^{\mathrm{sp}}$.
Accordingly, $\mathcal{G}_\alpha$ denotes the gradient mapping in \cref{eq:true-gradient-mapping} defined using this projection.

We first analyze the estimation error of the sphere-smoothed Jacobian estimator.
The following lemma gives the analogue of \citep[Lemma 8]{hikimaZerothorderGradientEstimators2025} for the response-map Jacobian.

\begin{lemma}
    \label{lem:estimated_Jacobian_error_sphere}
    Suppose that \cref{asm:main,asm:sphere_smoothing} hold.
    Let $\theta_t \in \Theta_\delta^{\mathrm{sp}}$, and let $\widehat{\pdv{\beta}{\theta}}(\theta_t)$ be defined by \cref{eq:sphere_jacobian_estimator}.
    Then,
    \begin{equation}
        \mathbb{E}\left[
            \norm{
                {\pdv{\beta}{\theta}}(\theta_t)
                -
                \widehat{\pdv{\beta}{\theta}}(\theta_t)
            }^2
            \right] \leq
        \frac{3\sigma_{\max}^2 n^2}{\delta^2 b_{\mathrm{sp}} b_{\min}}
        +
        3(L_{\beta,\theta}^{\mathrm{H}})^2\delta^4
        +
        \frac{(L_{\beta,\theta}^{\mathrm{H}})^2\delta^4 n^2}{6b_{\mathrm{sp}}}
        +
        \frac{18 n^2}{b_{\mathrm{sp}}(n+2)}
        \norm{{\pdv{\beta}{\theta}}(\theta_t)}_{\mathrm{F}}^2.
        \label{eq:sphere_jacobian_error_bound}
    \end{equation}
\end{lemma}
\begin{proof}
    Let $\bar{e}_j$ denote the $j$-th standard basis vector ($1 \leq j \leq m$) in the distribution parameter space $\mathbb{R}^m$.
    The proof is obtained by applying \citep[Lemma 8]{hikimaZerothorderGradientEstimators2025} component-wise to the vector-valued map $\beta$:
    \begin{align*}
                          & \mathbb{E}\left[
            \norm{
                {\pdv{\beta}{\theta}}(\theta_t)
                -
                \widehat{\pdv{\beta}{\theta}}(\theta_t)
            }^2\right] \leq \mathbb{E}\left[
            \norm{
                {\pdv{\beta}{\theta}}(\theta_t)
                -
                \widehat{\pdv{\beta}{\theta}}(\theta_t)
            }_{\mathrm{F}}^2
        \right]           &                  & (\norm{\cdot}_{\mathrm{F}} \text{ is Frobenius norm}) \\
        ={}               &
        \sum_{j=1}^m \mathbb{E}\left[
            \norm{
                {\pdv{\beta}{\theta}}(\theta_t)^\top \bar{e}_j
                -
                \frac{n}{b_{\mathrm{sp}}}
                \sum_{i=1}^{b_{\mathrm{sp}}}
                \frac{(\hat\beta_{t,i}^{+}-\hat\beta_{t,i}^{-})_j}{2\delta} u_{t}^{i}
            }^2
        \right]           &                  & (\text{\scalebox{0.9}{linearity of the expectation}})                                            \\
        \leq{}            &
        \sum_{j=1}^m \qty(
        \frac{3\sigma_{\max,j}^2 n^2}{\delta^2 b_{\mathrm{sp}} b_{\min}}
        +
        3(L_{\beta,\theta,j}^{\mathrm{H}})^2\delta^4
        +
        \frac{(L_{\beta,\theta,j}^{\mathrm{H}})^2\delta^4 n^2}{6b_{\mathrm{sp}}}
        +
        \frac{18 n^2}{b_{\mathrm{sp}}(n+2)}
        \norm{
            {\pdv{\beta}{\theta}}(\theta_t)^\top \bar{e}_j
        }^2
        )                 &                  & (\text{\scalebox{0.6}{\citep[Lemma 8]{hikimaZerothorderGradientEstimators2025}}}) \\
        \leq{}            &
        \frac{3\sigma_{\max}^2 n^2}{\delta^2 b_{\mathrm{sp}} b_{\min}}
        +
        3(L_{\beta,\theta}^{\mathrm{H}})^2\delta^4
        +
        \frac{(L_{\beta,\theta}^{\mathrm{H}})^2\delta^4 n^2}{6b_{\mathrm{sp}}}
        +
        \frac{18 n^2}{b_{\mathrm{sp}}(n+2)}
        \norm{
            {\pdv{\beta}{\theta}}(\theta_t)
        }_{\mathrm{F}}^2, &                  & (\text{by \cref{eq:componentwise_constants}})
    \end{align*}
    which concludes the proof.
\end{proof}

We now state and prove the formal version of \cref{prop:sphere_smoothing}.
\begin{proposition}[formal]
    Suppose that \cref{asm:main,asm:sphere_smoothing} hold.
    Consider the variant of \cref{alg:performative_prediction_1st_order} that uses the sphere-smoothed Jacobian estimator in \cref{eq:sphere_jacobian_estimator} and projection onto $\Theta_\delta^{\mathrm{sp}}$, with stepsize $0 < \alpha \leq 1/\Lsm$.
    Let $C_1$ be defined as in \cref{prop:bound_L1_L2}, and define
    \begin{equation*}
        C'_2 \coloneqq 3M_\ell I_{\max} \qty(\frac{3\sigma_{\max}^2}{b_{\min}} + 3(L_{\beta,\theta}^{\mathrm{H}})^2 + \frac{(L_{\beta,\theta}^{\mathrm{H}})^2}{6} + 18 (L_{\beta,\theta}^{\mathrm{Lip}})^2).
    \end{equation*}
    For any $\epsilon > 0$, choose $b_{\mathrm{sp}} = n b$ and $\delta = n^{1/6} b^{-1/6}$, initialize at $\theta_0 \in \Theta_\delta^{\mathrm{sp}}$, and choose $T$ and $b$ as the smallest integers satisfying
    \begin{equation}
        T \geq \frac{6\bigl(\mathcal{L}(\theta_0)-\mathcal{L}_*\bigr)} {\alpha\epsilon^2},
        \qquad
        b \geq \max\qty(
        b_{\min},
        n,
        \frac{3C_1}{\epsilon^2},
        \frac{n (3 C'_2)^{1.5}}{\epsilon^3}
        ).
        \label{eq:sphere_parameter_choice}
    \end{equation}
    Then the sphere-smoothed variant satisfies
    \begin{equation*}
        \frac{1}{T} \sum_{t=0}^{T-1} \mathbb{E}\left[\norm{\mathcal{G}_\alpha(\theta_t)}^2\right] \leq \epsilon^2,
    \end{equation*}
    and its total sample complexity is $\order{n^2 \epsilon^{-5}}$.
\end{proposition}
\begin{proof}
    By \cref{eq:bound_L1_L2_final}, we have
    \begin{align}
               & \mathbb{E}\left[
            \norm{
                \nabla\mathcal{L}(\theta_t)
                -
                \widehat{\nabla\mathcal{L}}(\theta_t)
            }^2
        \right] \notag                                                                                                                                                                                                                                                                                                                                                       \\
        \leq{} & 3\qty(\qty(L_{\ell,\theta}^{\mathrm{Lip}})^2 I_{\max} + \qty(L_{\beta,\theta}^{\mathrm{Lip}} L_{G,\beta}^{\mathrm{Lip}})^2) \mathbb{E}\left[\norm{\beta(\theta_t)-\hat\beta_t}^2\right] + 3 M_\ell I_{\max} \mathbb{E}\left[\norm{{\pdv{\beta}{\theta}}(\theta_t)-\widehat{\pdv{\beta}{\theta}}(\theta_t)}^2\right]. \label{eq:sphere_gradient_error_bound}
    \end{align}
    For the first term of \cref{eq:sphere_gradient_error_bound}, \cref{item:identifiability} and the definition of $C_1$ give
    \begin{equation}
        3\qty(\qty(L_{\ell,\theta}^{\mathrm{Lip}})^2 I_{\max} + \qty(L_{\beta,\theta}^{\mathrm{Lip}} L_{G,\beta}^{\mathrm{Lip}})^2) \mathbb{E}\left[\norm{\beta(\theta_t)-\hat\beta_t}^2\right] \le \frac{C_1}{b}.
        \label{eq:sphere_first_term_bound}
    \end{equation}
    For the second term of \cref{eq:sphere_gradient_error_bound}, applying \cref{lem:estimated_Jacobian_error_sphere} gives
    \begin{equation}
        \mathbb{E}\left[\norm{{\pdv{\beta}{\theta}}(\theta_t)-\widehat{\pdv{\beta}{\theta}}(\theta_t)}^2\right]
        \leq
        \frac{3\sigma_{\max}^2 n^2}{\delta^2 b_{\mathrm{sp}} b_{\min}}
        + 3(L_{\beta,\theta}^{\mathrm{H}})^2\delta^4
        + \frac{(L_{\beta,\theta}^{\mathrm{H}})^2\delta^4 n^2}{6b_{\mathrm{sp}}}
        + \frac{18 n^2}{b_{\mathrm{sp}}(n+2)}
        \norm{{\pdv{\beta}{\theta}}(\theta_t)}_{\mathrm{F}}^2.
        \label{eq:sphere_jacobian_error_bound_applied}
    \end{equation}
    We then bound the remaining Frobenius norm term.
    Since $\pdv{\beta}{\theta}$ is an $m \times n$ matrix, \cref{eq:Jacobian_norm_bounded_by_Lipschitz} gives
    \begin{equation}
        \norm{{\pdv{\beta}{\theta}}(\theta_t)}_{\mathrm{F}}^2 \leq \min(m,n) \norm{{\pdv{\beta}{\theta}}(\theta_t)}^2 \leq n (L_{\beta,\theta}^{\mathrm{Lip}})^2.
        \label{eq:Jacobian_Frobenius_norm_bound}
    \end{equation}
    Substituting \cref{eq:Jacobian_Frobenius_norm_bound} into \cref{eq:sphere_jacobian_error_bound_applied} yields
    \begin{equation}
        \mathbb{E}\left[\norm{{\pdv{\beta}{\theta}}(\theta_t)-\widehat{\pdv{\beta}{\theta}}(\theta_t)}^2\right]
        \leq
        \frac{3\sigma_{\max}^2 n^2}{\delta^2 b_{\mathrm{sp}} b_{\min}}
        + 3(L_{\beta,\theta}^{\mathrm{H}})^2\delta^4
        + \frac{(L_{\beta,\theta}^{\mathrm{H}})^2\delta^4 n^2}{6b_{\mathrm{sp}}}
        + \frac{18 (L_{\beta,\theta}^{\mathrm{Lip}})^2 n^3}{b_{\mathrm{sp}}(n+2)}.
        \label{eq:sphere_jacobian_error_bound_simplified}
    \end{equation}
    By setting $b_{\mathrm{sp}} = n b$ and $\delta = n^{1/6} b^{-1/6}$ and taking $b \geq n$, we can bound the second term as
    \begin{align}
                & 3M_\ell I_{\max} \mathbb{E}\left[\norm{{\pdv{\beta}{\theta}}(\theta_t)-\widehat{\pdv{\beta}{\theta}}(\theta_t)}^2\right]                                                                                                                                                  \notag                                                                                                                                                                                                      \\
        \leq {} & 3M_\ell I_{\max} \left( \frac{3\sigma_{\max}^2 n^2}{\delta^2 b_{\mathrm{sp}} b_{\min}} + 3(L_{\beta,\theta}^{\mathrm{H}})^2\delta^4 + \frac{(L_{\beta,\theta}^{\mathrm{H}})^2\delta^4 n^2}{6b_{\mathrm{sp}}} + \frac{18 (L_{\beta,\theta}^{\mathrm{Lip}})^2 n^3}{b_{\mathrm{sp}}(n+2)} \right)                                                                                               &  & (\text{by \cref{eq:sphere_jacobian_error_bound_simplified}})  \notag                \\
        \leq {} & 3M_\ell I_{\max} \left( \frac{3\sigma_{\max}^2 n^{2/3}}{b^{2/3} b_{\min}} + \frac{3(L_{\beta,\theta}^{\mathrm{H}})^2 n^{2/3}}{b^{2/3}} + \frac{(L_{\beta,\theta}^{\mathrm{H}})^2 n^{5/3}}{6b^{5/3}} + \frac{18 (L_{\beta,\theta}^{\mathrm{Lip}})^2 n}{b} \right)                                                                                                                             &  & (b_{\mathrm{sp}} = n b \text{ and }\delta = n^{1/6} b^{-1/6})    \notag \\
        = {}    & 3M_\ell I_{\max} \qty(\frac{3\sigma_{\max}^2}{b_{\min}} + 3(L_{\beta,\theta}^{\mathrm{H}})^2 + \frac{(L_{\beta,\theta}^{\mathrm{H}})^2}{6}\frac{n}{b} + 18 (L_{\beta,\theta}^{\mathrm{Lip}})^2 \qty(\frac{n}{b})^{1/3}) \frac{n^{2/3}}{b^{2/3}}                                                                                                                                       \notag                                                                                          \\
        \leq {} & 3M_\ell I_{\max} \qty(\frac{3\sigma_{\max}^2}{b_{\min}} + 3(L_{\beta,\theta}^{\mathrm{H}})^2 + \frac{(L_{\beta,\theta}^{\mathrm{H}})^2}{6} + 18 (L_{\beta,\theta}^{\mathrm{Lip}})^2) \frac{n^{2/3}}{b^{2/3}}                                                                                                                                                                                 &  & (n \leq b)                                                                          \\
        ={}     & \frac{n^{2/3} C'_2}{b^{2/3}}.
        \label{eq:sphere_second_term_bound}
    \end{align}
    Substituting \cref{eq:sphere_first_term_bound,eq:sphere_second_term_bound} into \cref{eq:sphere_gradient_error_bound} gives
    \begin{equation*}
        \mathbb{E}\left[
            \norm{
                \nabla\mathcal{L}(\theta_t)
                -
                \widehat{\nabla\mathcal{L}}(\theta_t)
            }^2
            \right] \leq \frac{C_1}{b} + \frac{n^{2/3} C'_2}{b^{2/3}}.
    \end{equation*}

    Thus, by the choices of $T$ and $b$ in \cref{eq:sphere_parameter_choice}, we can guarantee that
    \begin{equation*}
        \frac{1}{T} \sum_{t=0}^{T-1} \mathbb{E}\left[\norm{\mathcal{G}_\alpha(\theta_t)}^2\right] \leq \frac{2(\mathcal{L}(\theta_0) - \mathcal{L}_*)}{\alpha T} + \frac{C_1}{b} + \frac{n^{2/3} C'_2}{b^{2/3}} \leq \frac{\epsilon^2}{3} + \frac{\epsilon^2}{3} + \frac{\epsilon^2}{3} = \epsilon^2.
    \end{equation*}
    This is analogous to the bound in \cref{thm:main}.
    Since $b_{\mathrm{sp}} = n b$, the number of samples per iteration is
    \begin{equation*}
        b + 2 b_{\mathrm{sp}} b_{\min} = (2n b_{\min} + 1) b = \order{n^2 \epsilon^{-3}},
    \end{equation*}
    and hence the total sample complexity becomes
    \begin{equation*}
        T\qty(b+2b_{\mathrm{sp}} b_{\min})
        =
        \order{n^2 \epsilon^{-5}}.
    \end{equation*}
    This proves the formal version of \cref{prop:sphere_smoothing}.
\end{proof}
Compared with the coordinate-wise finite-difference estimator in \cref{thm:main}, the sphere-smoothed estimator improves the sample complexity from $\order{n^{2.5} \epsilon^{-5}}$ to $\order{n^2 \epsilon^{-5}}$.

\section{Additional Variant: Adaptive Stepsize Schedule}
\label{app:stepsize_schedule}

In the theoretical analysis, we employ a constant stepsize, which is theoretically preferable but may not be the most effective option in practice.
In numerical experiments, one may instead use adaptive optimizers such as Adam \citep{kingmaAdamMethodStochastic2017a} and AdamW \citep{loshchilovDecoupledWeightDecay2019}, which often exhibit better empirical performance.
Although we do not analyze Adam-type variants in this work, the error bound in \cref{prop:bound_L1_L2} may be useful for future convergence analyses.
\Cref{fig:experiments_result_adam} reports the corresponding Adam-style adaptive-stepsize results.
These results show that the proposed methods retain stable convergence under adaptive stepsizes across all experiments.

\afterpage{\begin{figure*}[t]
        \centering
        \begin{minipage}{0.75\textwidth}
            \begin{tabular}{@{}cc@{}}
                \includegraphics[width=0.48\linewidth]{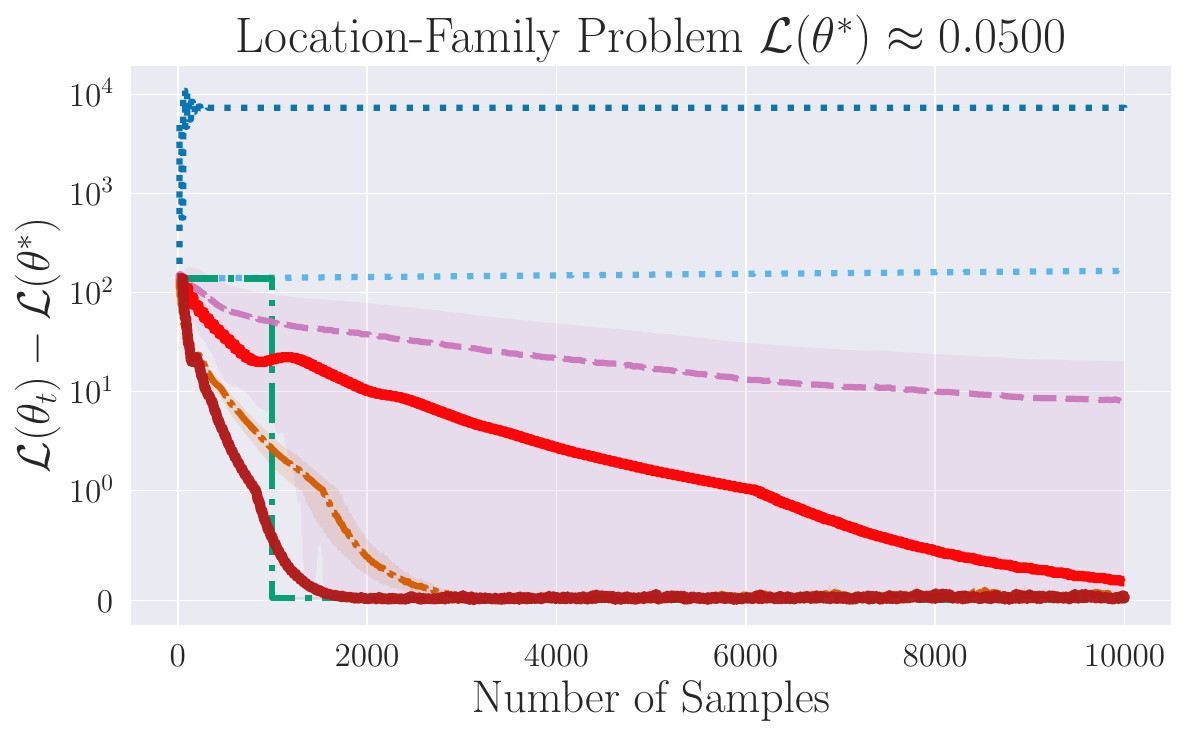}  &
                \includegraphics[width=0.48\linewidth]{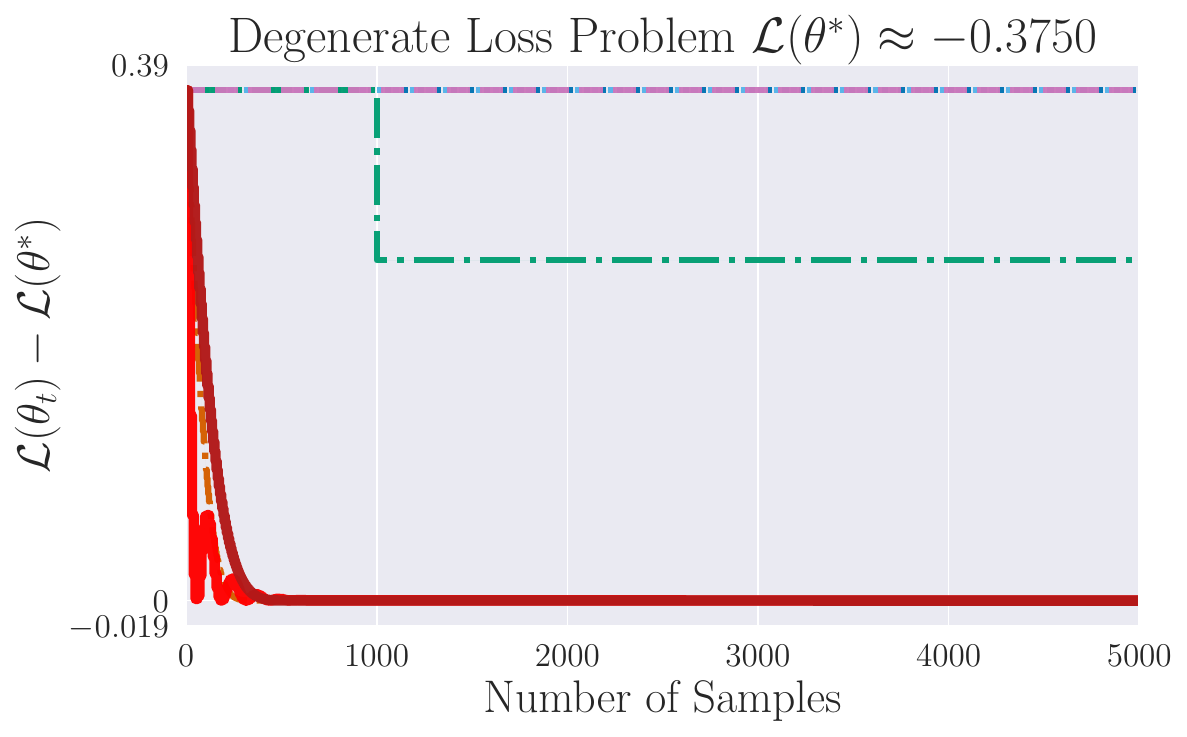} \\
                \includegraphics[width=0.48\linewidth]{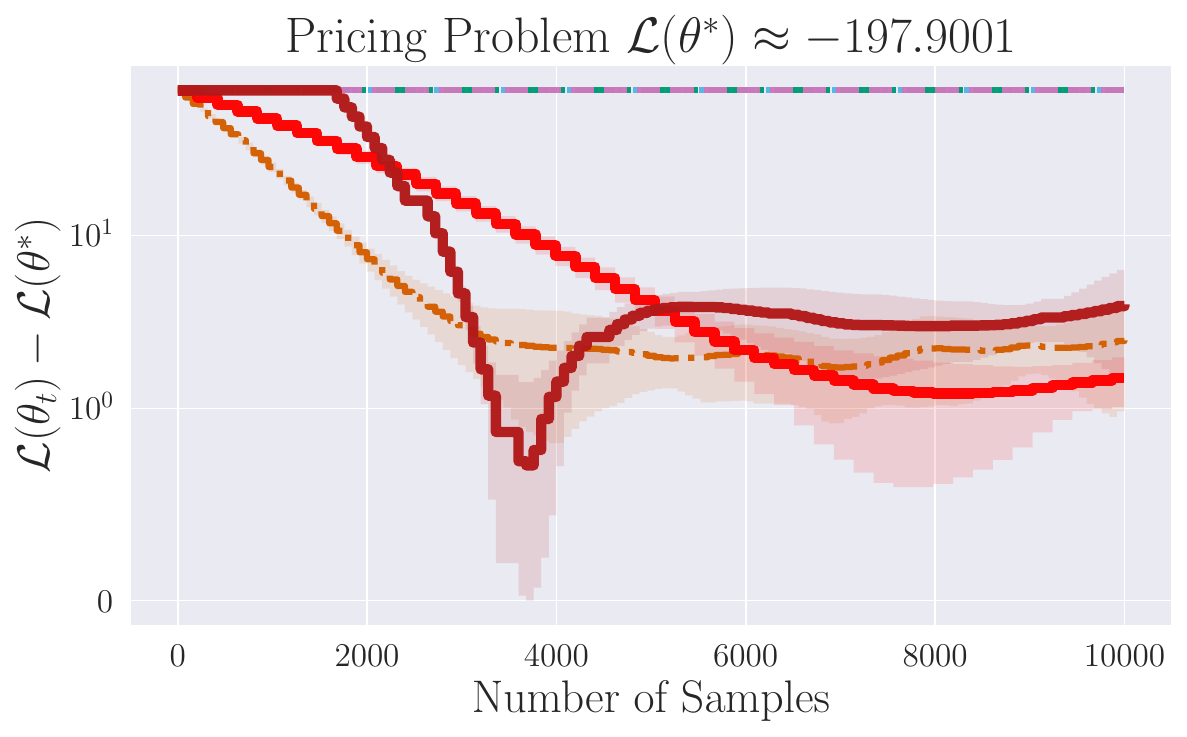} &
                \includegraphics[width=0.48\linewidth]{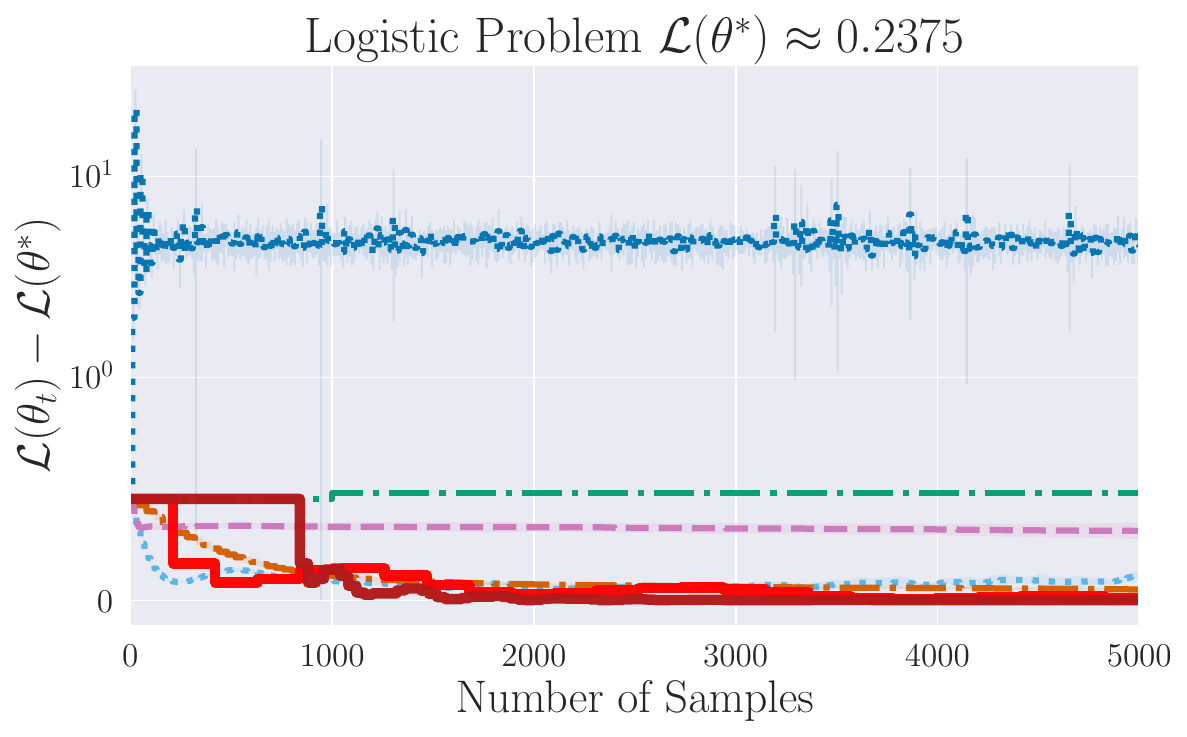}
            \end{tabular}
        \end{minipage}
        \hfill
        \begin{minipage}{0.2\textwidth}
            \includegraphics[width=\linewidth]{images/legend_pr.pdf}
        \end{minipage}
        \caption{Experimental results with Adam-style adaptive stepsizes.
            Solid curves show the mean excess performative risk, and shaded regions indicate one standard deviation.
            The vertical axis uses a symmetric logarithmic scale, which is linear for values between $-1$ and $1$ and logarithmic outside this range.
            The proposed methods retain stable convergence under adaptive stepsizes across all experiments.}
        \label{fig:experiments_result_adam}
    \end{figure*}}

\section{Experiment Details}

In this section, we provide details on the experimental setups and the problem setups.

\subsection{Experimental Setup}
\label{app:experimental_setup}

We first describe the parameter search used in the experiments.
We used a simple grid search to tune the algorithmic parameters.
For the stepsize parameter $\alpha$, we searched over
\begin{equation*}
    \alpha \in \{10^{-1},10^{-2},10^{-3}\}.
\end{equation*}
This parameter is the learning rate in Adam, and the remaining moment parameters were the PyTorch default values.
For the perturbation radius $\delta$, we searched over
\begin{equation*}
    \delta \in \{1,10^{-1},10^{-2}\}.
\end{equation*}
For the batch size $b$, we searched over
\begin{equation*}
    b \in \{d,2d,4d,8d\}.
\end{equation*}
For the update intervals of Proposed (cyclic), we searched over
\begin{equation*}
    K \in \{1,5,10\}.
\end{equation*}
Also, PerfGD used a history size $H=50$.
The plug-in method did not have tuned hyperparameters in these experiments because we used the same Gaussian surrogate specification in every setting.

Each hyperparameter configuration was evaluated using three random seeds with one-third of the maximum sample budget for each trial, and we selected the configuration with the lowest mean final performative risk across these trials.
After selection, we independently reran every trial to compute the reported mean and standard deviation.

The final hyperparameters used for each method and each experiment are listed in \cref{tab:hyperparameters_constant,tab:hyperparameters_adam}.
In the tables, $b$ denotes the batch size, $\alpha$ the stepsize or learning rate, $\delta$ the perturbation radius, $H$ the PerfGD history size, and $K$ the update interval for Proposed (cyclic).
\newcommand{\HyperparametersTableFont}{\scriptsize}
\newcommand{\HyperparametersCellFont}{\fontsize{5.5}{5.5}\selectfont}
\newcommand{\HyperparametersTabcolsep}{1pt}
\newcommand{\HyperparametersMethodColumnWidth}{0.08\textwidth}
\providecommand{\HyperparametersTableFont}{\small}
\providecommand{\HyperparametersCellFont}{\fontsize{8.2}{10.2}\selectfont}
\providecommand{\HyperparametersTabcolsep}{2pt}
\providecommand{\HyperparametersMethodColumnWidth}{0.125\textwidth}
\begin{table*}[t]
    \centering
    \caption{Final hyperparameters for experiments with constant.}
    \label{tab:hyperparameters_constant}
    {\HyperparametersTableFont
    \setlength{\tabcolsep}{\HyperparametersTabcolsep}
    \begin{tabularx}{\textwidth}{@{}p{\HyperparametersMethodColumnWidth} *{4}{>{\raggedright\arraybackslash\HyperparametersCellFont}X}@{}}
        \toprule
        Method & Location-family & Degenerate loss & Pricing & Logistic \\
        \midrule
        RRM & $b$ = 20 & $b$ = 2 & $b$ = 10 & $b$ = 10 \\
        RGD & $b$ = 40, $\alpha$ = 0.001 & $b$ = 2, $\alpha$ = 0.1 & $b$ = 10, $\alpha$ = 0.1 & $b$ = 40, $\alpha$ = 0.1 \\
        DFO & $b$ = 10, $\alpha$ = 0.001, $\delta$ = 1.0 & $b$ = 4, $\alpha$ = 0.01, $\delta$ = 1.0 & $b$ = 40, $\alpha$ = 0.01, $\delta$ = 1.0 & $b$ = 20, $\alpha$ = 0.1, $\delta$ = 1.0 \\
        PerfGD & $b$ = 5, $\alpha$ = 0.001, $H$ = 50 & $b$ = 2, $\alpha$ = 0.1, $H$ = 50 & $b$ = 10, $\alpha$ = 0.1, $H$ = 50 & $b$ = 10, $\alpha$ = 0.1, $H$ = 50 \\
        Proposed & $b$ = 5, $\alpha$ = 0.001, $\delta$ = 0.1 & $b$ = 2, $\alpha$ = 0.1, $\delta$ = 1.0 & $b$ = 20, $\alpha$ = 0.1, $\delta$ = 1.0 & $b$ = 10, $\alpha$ = 0.1, $\delta$ = 1.0 \\
        Proposed (cyclic) & $b$ = 5, $\alpha$ = 0.001, $\delta$ = 0.1, $K$ = 10 & $b$ = 2, $\alpha$ = 0.01, $\delta$ = 1.0, $K$ = 5 & $b$ = 80, $\alpha$ = 0.1, $\delta$ = 1.0, $K$ = 10 & $b$ = 10, $\alpha$ = 0.1, $\delta$ = 1.0, $K$ = 10 \\
        \bottomrule
    \end{tabularx}
    }
\end{table*}

\begin{table*}[t]
    \centering
    \caption{Final hyperparameters for experiments with adam.}
    \label{tab:hyperparameters_adam}
    {\HyperparametersTableFont
    \setlength{\tabcolsep}{\HyperparametersTabcolsep}
    \begin{tabularx}{\textwidth}{@{}p{\HyperparametersMethodColumnWidth} *{4}{>{\raggedright\arraybackslash\HyperparametersCellFont}X}@{}}
        \toprule
        Method & Location-family & Degenerate loss & Pricing & Logistic \\
        \midrule
        RRM & $b$ = 20 & $b$ = 2 & $b$ = 10 & $b$ = 10 \\
        RGD & $b$ = 40, $\alpha$ = 0.001 & $b$ = 2, $\alpha$ = 0.1 & $b$ = 10, $\alpha$ = 0.1 & $b$ = 10, $\alpha$ = 0.01 \\
        DFO & $b$ = 10, $\alpha$ = 0.1, $\delta$ = 0.1 & $b$ = 4, $\alpha$ = 0.01, $\delta$ = 1.0 & $b$ = 80, $\alpha$ = 0.1, $\delta$ = 1.0 & $b$ = 40, $\alpha$ = 0.01, $\delta$ = 1.0 \\
        PerfGD & $b$ = 10, $\alpha$ = 0.1, $H$ = 50 & $b$ = 2, $\alpha$ = 0.1, $H$ = 50 & $b$ = 10, $\alpha$ = 0.1, $H$ = 50 & $b$ = 10, $\alpha$ = 0.01, $H$ = 50 \\
        Proposed & $b$ = 5, $\alpha$ = 0.1, $\delta$ = 1.0 & $b$ = 2, $\alpha$ = 0.1, $\delta$ = 1.0 & $b$ = 10, $\alpha$ = 0.1, $\delta$ = 1.0 & $b$ = 10, $\alpha$ = 0.1, $\delta$ = 1.0 \\
        Proposed (cyclic) & $b$ = 5, $\alpha$ = 0.1, $\delta$ = 1.0, $K$ = 5 & $b$ = 2, $\alpha$ = 0.01, $\delta$ = 1.0, $K$ = 1 & $b$ = 80, $\alpha$ = 0.1, $\delta$ = 1.0, $K$ = 10 & $b$ = 40, $\alpha$ = 0.1, $\delta$ = 1.0, $K$ = 10 \\
        \bottomrule
    \end{tabularx}
    }
\end{table*}

All experiments were run on a Windows 11 Home 64-bit machine (version 10.0.26200, build 26200) with an Intel Core i7-1360P CPU (12 cores, 16 logical processors), \SI{16}{\giga\byte} of RAM.
The software environment used Python 3.11 and included NumPy 2.4.2 and PyTorch 2.10.0.

\subsection{Problem Setup}
\label{app:problem_setup}

We next describe the problem setups studied in the experiments.

\subsubsection{Location-Family Problem}

The first experiment studies a synthetic performative problem with a location-family distribution and squared loss.
In this setting, the data distribution follows a Gaussian distribution whose mean changes linearly in response to the deployed model parameter $\theta$, while the variance remains fixed.
This is a common toy problem in the performative prediction setting \citep{perdomoPerformativePrediction2020,izzoHowLearnWhen2021d}.

Let us state the problem formally.
For this problem, we set $d=n$.
We consider a model parameter $\theta\in\mathbb{R}^n$ and a data vector $z \in \mathbb{R}^{d}$ drawn from a Gaussian distribution, defined as
\begin{equation*}
    z \sim \mathcal{N}(\mu(\theta), \sigma^2 I_{d}), \qquad \mu(\theta) = M_0 + M_1 \theta,
\end{equation*}
where $M_0 \in \mathbb{R}^{d}$ is an intercept vector and $M_1 \in \mathbb{R}^{d \times n}$ is a linear operator.
The variance $\sigma^2 > 0$ is fixed and does not depend on $\theta$.
The loss function used is the squared loss
\begin{equation*}
    \ell(z;\theta) = \norm{z - \theta}^2.
\end{equation*}
The performative risk is therefore
\begin{align*}
    \mathcal{L}(\theta)
     & = \mathbb{E}_{z \sim \mathcal{N}(\mu(\theta), \sigma^2 I_{d})}[\norm{z - \theta}^2]                                                                                                      \\
     & = \mathbb{E}_{z \sim \mathcal{N}(\mu(\theta), \sigma^2 I_{d})}[\norm{z}^2] - 2\mu(\theta)^\top \theta + \norm{\theta}^2                                                                  \\
     & = d \sigma^2 + \norm{\mu(\theta) - \theta}^2                                                                            &  & (\mathbb{E}[z_i^2] = \mathrm{Var}[z_i] + \mathbb{E}[z_i]^2) \\
     & = d \sigma^2 + \norm{M_0 + M_1 \theta - \theta}^2.                                                                      &  & (\mu(\theta) = M_0 + M_1 \theta)
\end{align*}

In the reported experiments, we used $d=n=5$, $\sigma = 0.1$, and $\theta_0 = 0$.
The intercept vector $M_0$ and linear operator $M_1$ were randomly initialized from normal distributions with standard deviation $d$.
The constraint set was the Euclidean ball centered at the origin with radius $2\norm{\theta^*}$, where $\theta^*$ is the optimal model parameter.

\subsubsection{Degenerate Loss Problem}

The second experiment studies a two-dimensional example with a degenerate standard risk gradient, i.e., the first coordinate of $\nabla_1\mathcal{L}(\theta)$ vanishes identically.
For this problem, we set $d=n=2$.
Let $\theta=(x,y)\in[-R,R]^2$ and let $z\in\mathbb{R}^2$ follow
\begin{equation*}
    z \sim \mathcal{N}(\beta(\theta),\sigma^2 I_2),
    \qquad
    \beta(\theta) = \begin{pmatrix}
        x + a y + qx^2 \\
        0
    \end{pmatrix},
    \qquad
    q=\frac{1}{R}.
\end{equation*}
The loss function is
\begin{equation*}
    \ell(z;\theta)=z_1+\frac{\lambda}{2}y^2.
\end{equation*}
Therefore, the performative risk is
\begin{equation*}
    \mathcal{L}(\theta)=x+ay+qx^2+\frac{\lambda}{2}y^2.
\end{equation*}
By \cref{eq:nabla_L1,eq:nabla_L2}, the two terms in the gradient decomposition are
\begin{equation*}
    \nabla_1\mathcal{L}(\theta)   = \begin{pmatrix} 0 \\ \lambda y \end{pmatrix}, \qquad
    \nabla_2\mathcal{L}(\theta)   = \begin{pmatrix} 1+2qx \\ a \end{pmatrix}.
\end{equation*}
In the reported experiment, we used $R=1$, $a=0.5$, $\lambda=1.0$, and $\sigma^2=10^{-6}$.
The initial model parameter was $\theta_0=(0,0)$.
At this point, $\nabla_1\mathcal{L}(\theta_0)=(0,0)$, whereas $\nabla_2\mathcal{L}(\theta_0)=(1,a)\neq(0,0)$.

\subsubsection{Pricing Problem}

The third experiment studies a performative pricing problem where a seller sets prices $\theta \in \mathbb{R}^n$ to maximize expected revenue.
In this setting, the customers' demand (the quantity purchased) responds to the prices through a linear demand shift model.
The goal is to learn optimal prices while accounting for the demand response induced by the deployed prices.

Let us state the problem formally.
For this problem, we set $d=n$.
Let $\theta \in \mathbb{R}^n$ denote the price vector, and let $z \in \mathbb{R}^d$ denote the demand vector (quantity purchased at each price point).
The demand coordinates are assumed to be conditionally independent and to follow Poisson distributions that depend on the prices as
\begin{equation*}
    z_i \sim \operatorname{Poisson}(\lambda_i(\theta)), \qquad i=1,\ldots,d,
\end{equation*}
where the expected demand is
\begin{equation*}
    \lambda(\theta) = \mu_0 - \varepsilon \theta.
\end{equation*}
Here, $\mu_0 \in \mathbb{R}^d$ is the baseline demand and the scalar $\varepsilon>0$ is the common price sensitivity.

The goal is to maximize expected revenue, which is the inner product of prices and demand.
Equivalently, we minimize the negative revenue, so the loss function is
\begin{equation*}
    \ell(z; \theta) = -\theta^\top z,
\end{equation*}
which represents the negative revenue from selling at prices $\theta$ with demand $z$.
The prices are constrained to lie in a bounded box $\Theta = [0, R]^n$ where $R > 0$ is the maximum allowed price.

The performative risk is therefore
\begin{equation*}
    \mathcal{L}(\theta)
    = \mathbb{E}_{z \sim \mathcal{D}_{\lambda(\theta)}}[-\theta^\top z]
    = -\theta^\top \mathbb{E}[z]
    = -\theta^\top (\mu_0 - \varepsilon \theta),
\end{equation*}
where $\mathcal{D}_{\lambda(\theta)}$ denotes the product Poisson distribution and $\mathbb{E}[z] = \lambda(\theta)$.
This simplifies to
\begin{equation*}
    \mathcal{L}(\theta) = -\theta^\top \mu_0 + \varepsilon\norm{\theta}^2 = -\sum_{i=1}^n \theta_i \mu_{0,i} + \varepsilon\sum_{i=1}^n \theta_i^2.
\end{equation*}

The unconstrained minimizer can be found by solving the first-order condition:
\begin{equation*}
    {\pdv{\mathcal{L}}{\theta_i}}(\theta^{\mathrm{unc}}) = -\mu_{0,i} + 2\varepsilon \theta_i^{\mathrm{unc}} = 0,
\end{equation*}
which implies
\begin{equation*}
    \theta_i^{\mathrm{unc}} = \frac{\mu_{0,i}}{2\varepsilon}.
\end{equation*}
The constrained optimum is obtained by projecting this value onto $[0, R]$ coordinate-wise.

In the reported experiments, we used $d = n = 10$, $\varepsilon = 2.0$, and $\theta_0 = 5 \cdot \mathbf{1}_{10}$.
The baseline demand was sampled as
\begin{equation*}
    \mu_0 \sim \mathrm{Uniform}(12, 13),
\end{equation*}
sampled independently for each coordinate.
The feasible price set was
\begin{equation*}
    \Theta = [0, 5]^{10},
\end{equation*}
meaning prices are constrained between 0 and 5.

\subsubsection{Binary Classification with Logistic Loss} 

The fourth experiment studies a synthetic performative binary classification problem with logistic loss.
The setting is motivated by spam classification: each observation has a binary label, where the positive class represents spam and the negative class represents non-spam \citep{izzoHowLearnWhen2021d}.
The key performative feature is that the positive class distribution changes in response to the deployed classifier since spammers can adapt their behavior to evade detection, while the negative class distribution remains fixed since non-spammers do not have an incentive to change their behavior.
The Gaussian feature model is also consistent with prior empirical observations in NLP \citep{izzoHowLearnWhen2021d}.
Indeed, strong performance on a range of NLP tasks can be obtained by transforming standard BERT embeddings so that their empirical distribution resembles an isotropic Gaussian sample \citep{liSentenceEmbeddingsPretrained2020a}.

Let us state the problem formally.
For this problem, we set $d=n$.
Write the model parameter as $\theta = [\theta^{(0)}, \theta^{(1:n-1)}] \in \mathbb{R}^n$, where $\theta^{(0)}$ is the intercept and $\theta^{(1:n-1)}\in\mathbb{R}^{n-1}$ are the feature weights.
Each observation is $z=(x,y)$, where $x\in\mathbb{R}^{d-1}$ is the feature vector and $y\in\{0,1\}$ is the binary label.
In the spam interpretation, $y=1$ means that the email is spam, and $y=0$ means that the email is non-spam.
Our goal is to learn a linear classifier that predicts $y$ from $x$, and the performative aspect is that the positive-class distribution changes in response to the deployed classifier, while the negative-class distribution remains fixed.
We use the ridge-regularized cross-entropy loss
\begin{equation*}
    \ell(z;\theta)
    =
    -y\ln h\qty(\theta^{(0)} + \theta^{(1:n-1)\top} x)
    -(1-y)\ln(1-h\qty(\theta^{(0)} + \theta^{(1:n-1)\top} x))
    +\frac{\lambda}{2}\norm{\theta^{(1:n-1)}}^2
\end{equation*}
where $h(\cdot)$ is the logistic function defined as $h(x) = 1/(1+\exp(-x))$.
We assume that the label $y$ satisfies
\begin{equation*}
    y \sim \mathrm{Bernoulli}(\gamma),
\end{equation*}
where $0 < \gamma < 1$ is the prior probability of the positive class.
Conditional on the label, the feature vector is Gaussian with parameters $\mu_0 \in \mathbb{R}^{d-1}$, $\sigma_0 > 0$, and $\sigma_1 > 0$ as follows:
\begin{equation*}
    \begin{dcases}
        x \sim \mathcal{N}(\mu_0,\sigma_0^2 I_{d-1})         & \text{if } y=0, \\
        x \sim \mathcal{N}(\beta(\theta),\sigma_1^2 I_{d-1}) & \text{if } y=1.
    \end{dcases}
\end{equation*}
The distribution parameter mapping $\beta(\theta)$ is defined as
\begin{equation*}
    \beta(\theta) = \mu_1-\varepsilon\odot\theta^{(1:n-1)},
\end{equation*}
where $\odot$ denotes the Hadamard product and $\mu_1, \varepsilon \in \mathbb{R}^{d-1}$ are parameters.
This means that only the positive-class distribution changes in response to the deployed classifier, which models the spammers' incentive to adapt their behavior to evade detection.
The performative risk is therefore
\begin{align*}
    \mathcal{L}(\theta)
     & = \mathbb{E}_{z \sim \mathcal{D}_{\beta(\theta)}}[\ell(z;\theta)]                                                           \\
     & =
    \gamma
    \mathbb{E}_{x\sim\mathcal{N}(\beta(\theta),\sigma_1^2 I_{d-1})} \left[-\ln h\qty(\theta^{(0)} + \theta^{(1:n-1)\top} x)\right] \\
     & \quad +
    (1-\gamma)
    \mathbb{E}_{x\sim\mathcal{N}(\mu_0,\sigma_0^2 I_{d-1})} \left[-\ln(1-h\qty(\theta^{(0)} + \theta^{(1:n-1)\top} x))\right]      \\
     & \quad +
    \frac{\lambda}{2}\norm{\theta^{(1:n-1)}}^2.
\end{align*}
For a fixed $\theta$, we can approximate these Gaussian expectations deterministically using Gauss--Hermite quadrature \citep[p.890]{abramowitz1965handbook}.

In the reported experiment, we used $d = n = 10$.
This corresponds to 9 feature coordinates and one intercept coordinate.
The distributional parameters were
\begin{gather*}
    \mu_{0,i}\sim\mathrm{Uniform}(0.5,1.5),\qquad
    \mu_{1,i}\sim\mathrm{Uniform}(-1.5,-0.5),\\
    \sigma_0=\sigma_1=0.5,\qquad
    \varepsilon_i\sim\mathrm{Uniform}(2.5,3.5),\qquad
    \gamma=0.5.
\end{gather*}
The random vectors $\mu_0,\mu_1,\varepsilon\in\mathbb{R}^{9}$ were sampled independently across coordinates.
The ridge coefficient was
\begin{equation*}
    \lambda=10^{-2},
\end{equation*}
and the feasible parameter set was
\begin{equation*}
    \Theta=[-10,10]^{10}.
\end{equation*}
The initial model parameter $\theta_0$ was the zero vector.

\subsection{Statistical Significance Tests}
\label{app:significance_tests}

We report two-sided Wilcoxon signed-rank tests for the four constant-stepsize experiments shown in \cref{fig:experiments_result}.
At the maximum sample budget, we compared the final excess performative risk of Proposed (cyclic) with that of each other method, treating the latter methods as baselines.
For each comparison, trials were paired using the same random seed and problem instance for the two methods.
The null hypothesis was that the paired final-risk differences were centered at zero.
The tests rejected the null hypothesis at the $5\%$ level in almost all comparisons and favored Proposed (cyclic).
For example, Proposed (cyclic) significantly outperformed PerfGD on the degenerate-loss problem ($p = 0.0020$), RRM on the logistic problem ($p = 0.0020$), and RGD on the pricing problem ($p = 0.0020$).
Among the significant comparisons, the baseline attained the lower median final excess performative risk in two cases visible in \cref{fig:experiments_result}: DFO on the degenerate-loss problem ($p = 0.0273$) and Plugin on the location-family problem ($p = 0.0020$).
The only nonsignificant case was Proposed (cyclic) versus DFO on the pricing problem ($p = 0.3750$), where Proposed (cyclic) still had the lower median final excess performative risk.

\end{document}